\documentclass[12pt]{amsart}
\usepackage{amsmath}
\usepackage{libertine}
\usepackage[libertine]{newtxmath}
\usepackage{mathrsfs} 
\usepackage[T1]{fontenc} 
\usepackage{hyperref}
\usepackage{url}

\usepackage{lipsum}
\usepackage{latexsym}
\usepackage{cite,amsmath,amstext,amsbsy,bm} 
\usepackage{paralist}
\usepackage{datetime} 
\usepackage{colortbl}
\usepackage[dvipsnames]{xcolor}
\usepackage{mathtools}
\usepackage{euscript}
\usepackage[all]{xy}

\usepackage{graphicx} 
\usepackage[a4paper,text={150mm,250mm},centering]{geometry} 
\usepackage{comment}  
\usepackage{cite}  
\usepackage{thmtools}
\usepackage{enumitem}
\usepackage{letltxmacro} 
\usepackage{nameref}
\usepackage{cleveref}
\usepackage{calc}
\usepackage{interval}
\usepackage{manfnt}
\usepackage{tikz-cd}  
\usepackage[utf8]{inputenc}
\usepackage{newunicodechar} 
\usepackage[hyperpageref]{backref}

\usepackage{csquotes}

\usepackage{faktor} 

\theoremstyle{plain}
\newtheorem{thmx}{Theorem}
\renewcommand{\thethmx}{\Alph{thmx}}

\newtheorem{thm}{Theorem}[section]  
\newtheorem{lem}[thm]{Lemma}

\newtheorem{proposition}[thm]{Proposition}
\newtheorem{cor}[thm]{Corollary}

\newtheorem{corx}[thmx]{Corollary}

\theoremstyle{definition}
\newtheorem{dfn}[thm]{Definition}
\newtheorem{example}[thm]{Example}

\theoremstyle{remark}
\newtheorem{rem}[thm]{Remark}

\numberwithin{equation}{subsection}  

\theoremstyle{plain}
\newlist{thmlist}{enumerate}{1}
\setlist[thmlist]{wide = 0pt, labelwidth = 2em, labelsep*=0em, itemindent = 0pt, leftmargin = \dimexpr\labelwidth + \labelsep\relax, noitemsep,topsep = 1ex, font=\normalfont, label=(\roman*), ref=\thethm.(\roman{thmlisti})}

\addtotheorempostheadhook[thm]{\crefalias{thmlisti}{thm}}

\addtotheorempostheadhook[assumpsion]{\crefalias{thmlisti}{assumption}}

\addtotheorempostheadhook[prop]{\crefalias{thmlisti}{prop}}

\addtotheorempostheadhook[dfn]{\crefalias{thmlisti}{dfn}}

\addtotheorempostheadhook[lem]{\crefalias{thmlisti}{lem}}
\addtotheorempostheadhook[main]{\crefalias{thmlisti}{main}}

\addtotheorempostheadhook[rem]{\crefalias{thmlisti}{rem}}

\newlist{thmenum}{enumerate}{1} 
\setlist[thmenum]{wide = 0pt, labelwidth = 2em, labelsep*=0em, itemindent = 0pt, leftmargin = \dimexpr\labelwidth + \labelsep\relax, noitemsep,topsep = 1ex, font=\normalfont, label=(\roman*), ref=\thethmx.(\roman{thmenumi})}
\crefalias{thmenumi}{thmx}

\makeatletter
\newsavebox{\@brx}
\newcommand{\llangle}[1][]{\savebox{\@brx}{\(\m@th{#1\langle}\)}%
	\mathopen{\copy\@brx\kern-0.5\wd\@brx\usebox{\@brx}}}
\newcommand{\rrangle}[1][]{\savebox{\@brx}{\(\m@th{#1\rangle}\)}%
	\mathclose{\copy\@brx\kern-0.5\wd\@brx\usebox{\@brx}}}
\makeatother

\crefname{lem}{Lemma}{Lemmas}
\crefname{thm}{Theorem}{Theorems}
\crefname{proposition}{Proposition}{Propositions}
\crefname{dfn}{Definition}{Definitions}
\crefname{rem}{Remark}{Remarks}
\crefname{cor}{Corollary}{Corollaries}
\crefname{corx}{Corollary}{Corollaries}
\crefname{problem}{Problem}{Problems}
\crefname{thmx}{Theorem}{Theorems}
\crefname{claim}{Claim}{Claims}
\crefname{assumption}{Assumption}{Assumptions}

\def\ep{\varepsilon}

\def\rank{{\rm rank}}

\newcommand{\diae}{{}^\diamond\! E}

\makeatletter
\newcommand*{\rom}[1]{\expandafter\@slowromancap\romannumeral #1@}
\makeatother

\makeatletter     
\newcommand{\crefnames}[3]{%
	\@for\next:=#1\do{%
		\expandafter\crefname\expandafter{\next}{#2}{#3}%
	}%
}
\makeatother

\crefnames{part,chapter,section}{\S}{\S\S}

\newcommand{\sA}{\mathscr{A}}

\newcommand{\sC}{\mathscr{C}}
\newcommand{\sD}{\mathscr{D}}

\newcommand{\sI}{\mathscr{I}}

\newcommand{\cE}{\mathcal E}

\newcommand{\cF}{\mathcal F}

\newcommand{\cO}{\mathcal O}

\newcommand{\bB}{\mathbb{B}}
\newcommand{\bC}{\mathbb{C}}

\newcommand{\bP}{\mathbb{P}}

\newcommand{\bR}{\mathbb{R}}

\newcommand{\bZ}{\mathbb{Z}}

 \newcommand{\kg}{\mathfrak{g}} 
 
 \newcommand{\kk}{\mathfrak{k}} 
 \newcommand{\kl}{\mathfrak{l}}

 \newcommand{\ks}{\mathfrak{s}} 
 
 \newcommand{\ku}{\mathfrak{u}}

\def\db{\bar{\partial}}

 \def\d{\partial} 
\def\hess{\sqrt{-1}\partial\overline{\partial}}

\def\sn{\sqrt{-1}}

\def\nak{\text{\tiny  Nak}}

\def\End{\text{\small  End}}

\def\hom{\text{\small  Hom}}

\def\vol{\text{\small  Vol}}

\makeatletter 
\let\@wraptoccontribs\wraptoccontribs
\makeatother

\def\lmd{\lambda}

\makeatletter
\hypersetup{
	pdfauthor={\authors},
	pdftitle={\@title},
	pdfsubject={\@subjclass},
	pdfkeywords={\@keywords},
	pdfstartview={Fit},
	pdfpagelayout={TwoColumnRight},
	pdfpagemode={UseOutlines},
	bookmarks,
	colorlinks,
	linkcolor=Fuchsia,
	citecolor=NavyBlue,
	urlcolor=CarnationPink}
\makeatother

\begin{document} 
 	\title[characterization of  non-compact ball quotients]{A characterization of  complex quasi-projective \\manifolds uniformized by unit balls}

	\author{Ya Deng}  
\address{CNRS, Institut \'Elie Cartan de Lorraine, Universit\'e de Lorraine, F-54000 Nancy,
	France.}

\email{ya.deng@univ-lorraine.fr \quad ya.deng@math.cnrs.fr}

\urladdr{https://ydeng.perso.math.cnrs.fr} 
	
\dedicatory{With an appendix written jointly with   {\rm \textsc{Beno\^it Cadorel}}} 
\address[Beno\^it Cadorel]{Institut \'{E}lie Cartan de Lorraine, UMR 7502, Universit\'{e} de Lorraine, Site de Nancy, B.P. 70239, F-54506 Vandoeuvre-l\`{e}s-Nancy Cedex}
\email{benoit.cadorel@univ-lorraine.fr}
\thanks{This work is supported by \enquote{le fond  Chern} \`a l'IHES.}
\urladdr{http://www.normalesup.org/~bcadorel/} 

	\date{\today} 
	\begin{abstract}  
	In 1988 Simpson extended the Donaldson-Uhlenbeck-Yau theorem to the context of Higgs bundles, and as an application he proved a uniformization theorem which characterizes complex projective manifolds and quasi-projective curves whose universal coverings are complex unit balls.	In this paper we give a necessary and sufficient condition for quasi-projective manifolds to be uniformized by complex unit balls. This generalizes the uniformization theorem by Simpson.  
Several byproducts are also obtained in this paper.
\end{abstract}  
 
	\maketitle
	

\section{Introduction} 
\subsection{Main result}
The  main goal of this paper is to characterize complex  quasi-projective manifolds whose universal coverings are complex unit balls.
\begin{thmx}[=\cref{item1}]\label{main}
	Let $X$ be an $n$-dimensional complex projective manifold and let $D$ be a \emph{smooth} divisor on $X$ (which might contain several disjoint components). Let $L$ be an ample polarization on $X$. For the log Higgs bundle
	$(\Omega_X^1(\log D)\oplus \cO_X,\theta)$ on $(X,D)$ with the Higgs field $\theta$ defined by
	\begin{align}\label{eq:Higgs}
	\theta:\Omega_X^1(\log D)\oplus \cO_X&\to (\Omega_X^1(\log D)\oplus \cO_X)\otimes \Omega^1_X(\log D)\\\nonumber
	(a,b)&\mapsto (0,a),
	\end{align}
	 if it is $\mu_L$-polystable (see \cref{sec:stability} for the definition),  then one has the following inequality
\begin{align}\label{eq:BM}
	\big(2c_2(\Omega_X^1(\log D))-\frac{n}{n+1}c_1(\Omega_X^1(\log D))^2\big)\cdot c_1(L)^{n-2}\geq 0. 
\end{align}
When the   equality holds,     then  ${X-D}\simeq \faktor{\bB^n}{\Gamma}$ for some  torsion  free lattice $\Gamma\subset PU(n,1)$ acting on $\bB^n$. Moreover, $X$ is the (unique)   toroidal compactification of $\faktor{\bB^n}{\Gamma}$, and each connected component of $D$ is the \emph{smooth}  quotient of an Abelian variety $A$  by a finite group acting  freely on $A$. 
\end{thmx} 
Let us stress here that the \emph{smoothness} of $D$ in \cref{main} is indeed necessary  if one would like to characterize non-compact ball quotients: in \cref{item2} we prove that the universal cover of $X-D$ is not the complex unit ball $\bB^n$ if $D$ is assumed to be simple normal crossing but not smooth, leaving other conditions in \cref{main} unchanged.  Thus, it might be more appropriate to say that in this paper we give a  characterization of  \emph{smooth toroidal compactification} of non-compact ball quotients.


Note that when $D$ is empty or  when $\dim\, X=1$, \cref{main} has already been proved by Simpson   \cite[Proposition 9.8]{Sim88}.  As we will see later, we follow his strategy closely to prove the above theorem.  Let us also mention that the inequality \eqref{eq:BM} is a direct consequence  of Mochizuki's deep work on the Bogomolov-Gieseker inequality for parabolic Higgs bundles \cite[Theorem 6.5]{Moc06}. Our main contribution   is  the uniformization result when the equality in \eqref{eq:BM} is achieved. The proof builds on Simpson's ingenious ideas \cite{Sim88} on characterizations of complete varieties uniformized by Hermitian symmetric spaces, as well as  Mochizuki's celebrated work on Simpson correspondence for tame harmonic bundles \cite{Moc06}. Since the  Kobayashi-Hitchin correspondence for general slope polystable parabolic Higgs bundles is still unproven, we need some additional methods to prove the above uniformization result (see \cref{sec:difficulty} for  rough ideas).

We will show that the conditions in \cref{main} is indeed necessary, by proving the following slope stability (with respect to a more general polarization) result  for the natural log Higgs bundles associated to toroidal compactification of non-compact ball quotient by torsion free lattice.
\begin{thmx}[=\cref{sec:toroidal}]\label{main2}
  Let $\Gamma\subset PU(n,1)$ be a torsion free lattice  with only unipotent parabolic elements.  Let $X$ be  the (smooth) toroidal compactification of the ball quotient  $\faktor{\bB^n}{\Gamma}$. Write  $D:=X-\faktor{\bB^n}{\Gamma}$ for the boundary divisor, which is a disjoint union of Abelian varieties.   Let $\alpha\in H^{1,1}(X,\bR)$ be a big and nef cohomology $(1,1)$-class on $X$ containing a positive closed $(1,1)$-current $T\in \alpha$ so that $T|_{X-D}$ is a smooth K\"ahler form  and has at most \emph{Poincar\'e   growth} near $D$ (for example, $\alpha=c_1(K_X+D)$ or $\alpha$ contains a K\"ahler form $\omega$).   Then  one has the following equality for Chern classes
	\begin{align}\label{eq:equality}
	2c_2(\Omega_X^1(\log D))-\frac{n}{n+1}c_1(\Omega_X^1(\log D))^2=0. 
	\end{align} 
The log Higgs bundle
	$(\Omega_X^1(\log D)\oplus \cO_X,\theta)$ defined in \eqref{eq:Higgs} is $\mu_\alpha$-polystable for the above big and nef polarization $\alpha$. In particular, it is slope polystable with respect to any K\"ahler polarization and the polarization by the big and nef class $c_1(K_{X}+D)$. 
\end{thmx}  
Since both stability of log Higgs bundles and Chern  equality \eqref{eq:equality} are invariant under taking conjugates with respect to the Galois action,  a direct consequence of \cref{main,main2} is  the following rigidity result of   ball quotient under  the automorphism of  complex number field $\bC$   to its coefficients of defining equations.
\begin{corx}\label{corx}
	Let $\Gamma\subset PU(n,1)$ be a torsion free lattice, and let $X:=\faktor{\bB^n}{\Gamma}$ be the ball quotient, which carries a unique algebraic structure, denoted by $X_{\rm alg}$. For any automorphism $\sigma\in {\rm Aut}(\bC)$, let $X^{\sigma}_{{\rm alg}}:=X_{\rm alg}\times_{\sigma}{\rm Spec}(\bC)$ be the conjugate  variety of $X_{\rm alg}$   under  the automorphism $\sigma$, and denote by $X^{\sigma}$ the analytification of $X^{\sigma}_{{\rm alg}}$. Then $X^\sigma$ is also a ball quotient, namely there is  another torsion free lattice $\Gamma^\sigma\subset PU(n,1)$ so that $X^\sigma=\faktor{\bB^n}{\Gamma^\sigma}$.
\end{corx} 
When $\Gamma$ is arithmetic, \cref{corx} has been proved by  Kazhdan \cite{Kaz83}. When $\Gamma$ is non-arithmetic, it was     proved by Mok-Yeung \cite[Theorem 1]{MY93} and by Baldi-Ullmo \cite[Theorem 8.4.2]{BU20}.


In this paper we   obtain some byproducts, and let us mention a few.  We prove the Simpson-Mochizuki correspondence for principal system of log Hodge bundles over projective log pairs (see \cref{Tannakian}).  We  give a characterization of slope stability with respect to big and nef classes for  log Higgs bundles on   K\"ahler log pairs (see \cref{thm:criteria}). 
We also give a very simple proof of the negativity of kernels of Higgs fields of tame harmonic bundles by Brunebarbe \cite{Bru17} (originally by Zuo \cite{Zuo00} for system of log Hodge bundles), using some extension theorems of plurisubharmonic functions in complex analysis (see \cref{thm:negativity}). In the appendix written jointly with Beno\^it Cadorel, we prove a metric rigidity result for toroidal compactification of non-compact ball quotients (see \cref{thm:rigidity}). 
\subsection{A few histories} 
Since the main purpose of this paper is to prove the uniformization result  rather than the Miyaoka-Yau type inequality \eqref{eq:BM}, we shall only recall some earlier work related to the characterization of ball quotient, and we refer the readers to \cite{GKT16,GT16}  for more references on the  Miyaoka-Yau type inequalities.

Based on his proof of the Calabi conjecture \cite{Yau78}, Yau established   the inequality \eqref{eq:BM} when $X$ is a projective manifold and $D=\varnothing$ with $K_X$ ample.  He proved that $X$ is uniformized by the complex unit ball in   case of equality. 
Miyaoka-Yau inequality and uniformization result were extended to the context of compact
K\"ahler varieties with quotient singularities by Cheng-Yau \cite{CY86} using orbifold
Kahler-Einstein metrics. 
 A partial uniformization result for {smooth minimal models of general type} have been obtained by 
Zhang \cite{Zha09}. More recently, uniformization result has been   extended to projective varieties with klt singularities in the series of work \cite{GKPT19,GKPT19b} by Greb-Kebekus-Peternell-Taji.

All the above works dealt with \emph{compact}  varieties. A strong uniformization result was established by Kobayashi \cite{Kob84,Kob85}  in the case of \emph{open orbifold surfaces} (see also \cite{CY86}).  In \cite{CY86} Cheng-Yau  also gave a differential geometric characterization of quasi-projective ball
quotients of any dimensions using the method of bounded geometry in \cite{CY80}. At almost the same time, based on \cite{CY86},  Tian-Yau \cite{TY87} and Tsuji \cite{Tsu88}  independently established similar \emph{algebraic geometric} characterizations of non-compact  ball quotient of any dimension.  
 See also \cite{KNS89,Kob90,Yau93} for more related works on uniformization results. 

All these aforementioned uniformization results are built on the positivity of the (log) canonical sheaf of the varieties together with existence of K\"ahler-Einstein metrics. In \cite{Sim88}, Simpson established a remarkable uniformization result in terms of stability of Higgs bundles. We essentially follow his approaches in this paper. In next subsection, we shall recall his ideas and discuss  main difficulties in generalizing his methods to the context of \emph{non-compact 
  varieties}.
\subsection{Main strategy}\label{sec:difficulty}\label{sec:main}
We mainly follow Simpson's strategy \cite{Sim88} to prove \cref{main}. 
Let us explain our rough ideas in the proof of \cref{main} when the equality in \eqref{eq:BM} holds. 

\noindent {\bf Step 1.} Following Simpson in the compact setting, we first define  systems of log Hodge bundles over   log pairs. We   prove that,  a system of log Hodge bundles on a projective log pair with vanishing first and second Chern classes admits an \emph{adapted} Hodge metric. The proof is based on Mochizuki's celebrated theorem \cite[Theorem 9.4]{Moc06} on the existence of  harmonic metric, and   $\bC^*$-action invariant property of log Hodge bundles. 

\noindent {\bf Step 2.}  We generalize the result in Step 1 to the context of \emph{principal bundles}. Fix a Hodge group $G_0$. Following Simpson again, we  define a principal system of log Hodge bundles $(P,\tau)$ on log pairs $(X,D)$ with the structure group $K\subset G$, where $G$ is the complexification of $G_0$. Based on the result in Step 1 together with some similar Tannakian arguments in \cite{Sim90}, in \cref{Tannakian} we prove that if there is a faithful Hodge representation $\rho:G\to GL(V)$ for some polarized Hodge structure $(V=\oplus_{i+j=w}V^{i,j},h_V)$ so that the system of log Hodge bundles $(P\times_KV,d\rho(\tau))$ is $\mu_L$-polystable  with $\int_X ch_2(P\times_KV)\cdot c_1(L)^{\dim X-2}=0$, then there is a metric  reduction $P_H$ for $P|_{X-D}$     so that the triple $(P|_{X-D},\tau|_{X-D},P_H)$ gives rise to a \emph{principal variation of Hodge structures} on $X-D$.

\noindent {\bf Step 3.}    For the system of log Hodge bundles $(E:=\Omega_X^1(\log D)\oplus \cO_X,\theta)$ in \cref{main},    we first associate it a  principal system of log Hodge bundles $(P,\tau)$ in \cref{prop:converse}, whose Hodge group $G_0=PU(n,1)$ is of Hermitian type (see \cref{def:hermitian}). One can easily show that $c_2(P\times_K\kg)=c_2(\End(E)^\perp)=0$ when the equality in \eqref{eq:BM} holds, where $\End(E)^\perp$ denotes the trace free part of $\End(E)$.   By  \cref{prop:stability}, the system of log Hodge bundles  $(P\times_K\kg,d(Ad)(\tau))=(\End(E)^\perp,\theta_{End(E)^\perp})$ is also  slope polystable if $(E,\theta)$ is slope polystable. Since the adjoint representation $Ad:G\to GL(\kg)$ is a faithful Hodge representation, by the result in Step 2, there is a metric  reduction $P_H$ for $P|_{X-D}$     so that the triple $(P|_{X-D},\tau|_{X-D},P_H)$ gives rise to a \emph{principal variation of Hodge structures} on $X-D$. Since $\tau:T_X(-\log D)\to P\times_K\kg^{-1,1}$ is an isomorphism, this implies that the period map $p:\widetilde{X-D}\to \faktor{PU(n,1)}{U(n)}$ associated to $(P|_{X-D},\tau|_{X-D},P_H)$ from the universal cover $\widetilde{X-D}$ of $X-D$ to the period domain $\faktor{G_0}{K_0}=\faktor{PU(n,1)}{U(n)}$ is \emph{locally biholomorphic}.   For more details, see Step one of the proof of \cref{thm:equality}.

\noindent {\bf Step 4.} We have to prove that the period map $p$ in Step 3 is moreover a biholomorphism. Note that when $D=\varnothing$, this step is quite easy.  In \cref{rem:complete} we show that it suffices to prove that the hermitian metric $\tau^*h_H$ on $X-D$ is   \emph{complete}, where $h_H$ is the hermitian metric on $P\times_K\kg^{-1,1}|_{X-D}$ induced by the metric reduction $P_H$ together with the Killing form of $\kg$. This step is slightly involved and the readers can find it in Step two of the proof of \cref{thm:equality}. To be brief, we establish a precise model  metric (ansatz) for $(E,\theta)\otimes (E^*,\theta^*)$ locally around $D$ with at most log growth, and we prove that this local metric and $h_H$ are mutually 
bounded by one another using similar ideas in \cite[\S 4]{Sim90}. Based on this model metric, we obtain a precise norm estimates for $h_H$ near $D$, so that we can prove  that $\tau^*h_H$ is a complete metric on $X-D$.  This concludes that the universal cover of $X-D$ is the  unit ball $\faktor{PU(n,1)}{U(n)}$. 
 \subsection{Acknowledgments} This work owes a lot to the deep work   \cite{Sim88,Sim90,Sim92,Moc06}, to which I express my deepest gratitude. I sincerely thank  Professor  Carlos Simpson  for answering my questions, as well as his suggestions and    encouragements. I  thank Professor Takuro Mochizuki for sending me his personal notes on the proof of \cref{prop:stability}. I also thank Professors Jean-Pierre Demailly, Henri Guenancia, Emmanuel Ullmo, Shing-Tung Yau,   and Gregorio Baldi, Jiaming Chen, Chen Jiang, Jie Liu, Mingchen Xia for very helpful discussions and their remarks on this paper. My special thanks go to Beno\^it Cadorel for his very fruitful discussions on the toroidal compactification, which lead  to a joint appendix with him in this paper.  Last but not least, I am grateful to the referee  for his/her  careful readings and very helpful
 comments to improve this manuscript. 
 \section{Log Higgs bundles and system of log Hodge bundles}
\subsection{Higgs bundles and tame harmonic bundles}
	In this section we   recall the   definition of Higgs bundles and tame harmonic bundles. We refer the readers to  \cite{Sim88,Sim90,Sim92,Moc02,Moc07} for further details.
\begin{dfn}\label{Higgs}
	Let $X$ be a complex manifold. A \emph{Higgs bundle} on $X$  is a pair $(E,\theta)$ where $E$ is a holomorphic vector bundle with $\db_E$ its complex structure, and $\theta:E\to E\otimes \Omega^1_X$ is a holomorphic one form with value in $\End(E)$, say \emph{Higgs field},  satisfying $\theta\wedge\theta=0$.
\end{dfn}

Let  $(E,\theta)$ be a Higgs bundle  over a complex manifold $X$.  A smooth hermitian metric $h$ of $E$ is called \emph{harmonic} if   $D_h:=d_h +\theta+\overline{\theta}_h$ is flat. Here $d_h$ is the Chern connection of $(E,h)$, and $\overline{\theta}_h$ is the adjoint of $\theta$ with respect to $h$. 
\begin{dfn}[Harmonic bundle] A harmonic bundle on a complex manifold $X$ is triple $(E,\theta,h)$ where $(E,\theta)$ is 
	a Higgs bundle and $h$ is  a  harmonic metric for $(E,\theta)$.
\end{dfn}
A \emph{log pair} consists of an   $n$-dimensional complex manifold $X$,  and 
a simple normal crossing divisor $D$ on $X$.
\begin{dfn}(Admissible coordinate)\label{def:admissible} Let $p$ be a point of $X$, and assume that $\{D_{j}\}_{ j=1,\ldots,\ell}$ 
	be components of $D$ containing $p$. An \emph{admissible coordinate} around $p$
	is the tuple $(U;z_1,\ldots,z_n;\varphi)$ (or simply  $(U;z_1,\ldots,z_n)$ if no confusion arises) where
	\begin{itemize}
		\item $U$ is an open subset of $X$ containing $p$.
		\item there is a holomorphic isomorphism   $\varphi:U\to \Delta^n$ so that  $\varphi(D_j)=(z_j=0)$ for any
		$j=1,\ldots,\ell$.
	\end{itemize} 
	We shall write $U^*:=U-D$,   $U(r):=\{z\in U\mid |z_i|<r, \, \forall i=1,\ldots,n\}$ and $U^*(r):=U(r)\cap U^*$.  
\end{dfn}
Recall that the Poincar\'e metric $\omega_P$ on $(\Delta^*)^\ell\times \Delta^{n-\ell}$ is described as 
$$
\omega_P=\sum_{j=1}^{\ell}\frac{\sqrt{-1}dz_j\wedge d\bar{z}_j}{|z_j|^2(\log |z_j|^2)^2}+\sum_{k=\ell+1}^{n} \frac{\sqrt{-1}dz_k\wedge d\bar{z}_k}{(1-|z_k|^2)^2}. 
$$ 
\begin{dfn}[Poincar\'e growth]\label{def:Poincare}
Let $(X,D)$ be a log pair.  A   hermitian metric $\omega$ on $X-D$ has at most (resp. the same) \emph{Poincar\'e growth  near $D$}   if for any point $x\in D$, there is an admissible coordinate $(U;z_1,\ldots,z_n)$ centered at $x$ and a constant $C_U>0$ so that $\omega\leq C_U\omega_P$    (resp. $\omega\sim \omega_P$)  holds  over $U^*(r)$ for some $0<r<1$.   
\end{dfn}

\begin{rem}[Global K\"ahler metric with Poincar\'e growth]\label{rem:Poincare}
Let   $(X,\omega)$ be a compact K\"ahler  manifold and $D=\sum_{i=1}^\ell D_i$ is a simple normal crossing divisor on $X$. By Cornalba-Griffiths \cite{CG75}, one can construct a \emph{K\"ahler current} $T$ over $X$, whose restriction on $X-D$ is a complete K\"ahler form, which has the same Poincar\'e growth near $D$ as follows.

   Let $\sigma_i$ be the section $H^0(X,\cO_X(D_i))$ defining $D_i$, and we pick any smooth metric $h_i$ for the line bundle $\cO_X(D_i)$.  One can prove that the closed $(1,,1)$-current  \begin{align}\label{eq:metric}
T:=\omega-\hess \log (-\prod_{i=1}^{\ell}\log  |\ep\cdot \sigma_i|_{\cdot h_i}^2),
\end{align} 
 the desired K\"ahler current when $0<\ep\ll 1$.
\end{rem}

\subsection{Log Higgs bundle and adapted harmonic metrics}
Throughout this paper, we mainly consider  \emph{log Higgs bundles} $(E,\theta)$ over  log pairs.
\begin{dfn}[Log Higgs bundles]
Let $(X,D)$ be a log pair. A log Higgs bundle consists of a pair $(E,\theta)$ with  $E$ a holomorphic vector bundle on $X$ and $\theta:E\to E\otimes \Omega_X^1(\log D)$ with $\theta\wedge\theta=0$.  
\end{dfn}

\begin{dfn}[Adapted harmonic metric]\label{def:adapted}
	Let $(X,D)$ be a log pair, and let $(E,\theta)$ be a log Higgs bundle on $(X,D)$. Suppose that $h$  is a harmonic metric for the Higgs bundle  $(E,\theta)|_{X-D}$. It is called    \emph{adapted to} the log Higgs bundle if for any admissible coordinate $(U;z_1,\ldots,z_n)$ and any $(a_1,\ldots,a_\ell)\in [0,1)^{\ell}$, one has
	$$
	E(U)=\{\sigma\in E(U-D)\mid |\sigma|_h=\cO(\prod_{i=1}^{\ell}|z_i|^{-a_i-\varepsilon})\ \mbox{ for any }\ep>0  \}
	$$ 
\end{dfn}
In the terminology of \cite{Moc06,Moc07}, the above definitions are equivalent  that $(E,\theta)$ is a parabolic Higgs bundle  with trivial parabolic structures over $(X,D)$ of weight $(0,\ldots,0)$, and the harmonic bundle $h$ for $(E,\theta)|_{X-D}$  is adapted to its parabolic structures.
  
\subsection{Slope stability}\label{sec:stability}
	Let $(X,\omega)$ be a compact K\"ahler manifold of dimension $n$ and let $D$ be a simple normal crossing divisor on $X$. Let $(E,\theta)$ be a log Higgs bundle on $(X,D)$.  Let $\alpha$ be a big and nef cohomology $(1,1)$-class on $X$. For any  torsion free coherent sheaf $F$, its \emph{degree with respect to $\alpha$} is defined by $\deg_\alpha(F):=c_1(F)\cdot \alpha^{n-1}$, and its \emph{slope with respect to $\alpha$} is defined by
 $\mu_{\alpha}(F):=\frac{\deg_\alpha(F)}{\rank\, F}$.  Consider a log Higgs bundle $(E,\theta)$ on $(X,D)$. A \emph{Higgs sub-sheaf} is a  saturated coherent
	torsion free subsheaf $E'\subset E$ so that $\theta(E')\subset E'\otimes \Omega^1_X(\log D)$.   We say $(E,\theta)$ is \emph{$\mu_\alpha$-stable}  if  for Higgs sub-sheaf $E'$ of $E$, with $0<\textnormal{rank}\,
E'<\textnormal{rank}\, E$, the condition 
 $  \mu_{\alpha}(E')<\mu_\alpha(E)$ 
is satisfied. $(E,\theta)$  is $\mu_\alpha$-\emph{polystable} if it is a direct sum of $\mu_\alpha$-stable log Higgs bundles with the same slope. 

When $\alpha=\{\omega\}$ where $\omega$ is a K\"ahler form on $X$, we write $\mu_\omega$ instead of $\mu_{\alpha}$. When $\alpha=c_1(L)$ for some    ample line bundle $L$  on $X$, we use the notation $\mu_L$ instead of $\mu_\alpha$.

By Simpson \cite{Sim90}, there is a $\bC^*$-action on log Higgs bundles $(E,\theta)$ defined by $(E,t\theta)$ for any $t\in \bC^*$. It follows from the definition that, if $(E,\theta)$ is $\mu_\alpha$-stable, then $(E,t\theta)$ is also $\mu_\alpha$-stable for any $t\in \bC^*$.

The following celebrated \emph{Simpson correspondence for tame harmonic bundles} proved by Mochizuki \cite{Moc06} is a crucial ingredient in this paper.  
\begin{thm}[Mochizuki]\label{thm:mochizuki} 
Let $(X,D)$ be a projective log pair endowed with an ample polarization $L$. A log Higgs bundle $(E,\theta)$ on $(X,D)$ is $\mu_L$-polystable with $\int_Xc_1(E)\cdot c_1(L)^{\dim X-1}=\int_X ch_2(E)\cdot c_1(L)^{\dim X-2}=0$ if and only if there is a  harmonic metric $h$ for $(E|_{X-D},\theta|_{X-D})$ adapted to $(E,\theta)$. When $(E,\theta)$ is moreover stable, such a harmonic metric $h$ is unique up to some positive constant multiplication.
\end{thm}
Let us mention that in \cite{Biq97} Biquard has proved a stronger theorem when the divisor $D$ in \cref{thm:mochizuki} is smooth.  

The poly-stability is also preserved under tensor product and dual by Mochizuki \cite[Proposition 4.10]{Moc20}.
\begin{thm}[Mochizuki]\label{prop:stability}
Let $(X,D)$ be a projective log pair endowed with an ample polarization $L$.	Let $(E,\theta)$ be a $\mu_L$-polystable log Higgs bundle on $(X,D)$. Then the tensor product $T^{a,b}(E,\theta)$ is still a $\mu_L$-polystable log Higgs bundle for $a,b\in \bZ_{\geq 0}$.  Here $T^{a,b}(E,\theta):=\big(\hom(E^{\otimes a},E^{\otimes b}),\theta_{a,b}\big)$ is the induced log Higgs bundle by taking the tensor product.
\end{thm}
Since  \cite[Proposition 4.10]{Moc20} worked with the much more general case than what we need, we shall provide a quick proof for \cref{prop:stability} for completeness sake. The idea essentially follows   \cite[Corollary 3.8]{Sim92} in the compact setting.
\begin{proof}[Proof of \cref{prop:stability}] 
By the Mehta-Ramanathan type theorem proved by Mochizuki \cite[Proposition 3.29]{Moc06}, $T^{a,b}(E,\theta)$ is $\mu_L$-polystable if and only if $T^{a,b}(E,\theta)|_{Y}$  is $\mu_L$-polystable, where $Y$ denotes a complete intersection of sufficiently ample general hypersurfaces in $X$. This enables us to reduce the desired statement to the case of curves. Assume now that $\dim  X=1$. By \cite{Sim90} or \cite[Th\'eor\`eme 8.1]{Biq97}, $(E,\theta)|_{X-D}$ admits a Hermitian-Yang-Mills metric $h$:  $$\Lambda_{\omega}F_h(E)=\lambda\otimes \vvmathbb{1}_E,$$ where $\omega$ is some K\"ahler form in $c_1(L)$, and $\lambda$ is some topological constant.  Moreover, $h$ is  adapted to  $(E,\theta)$, and is adapted to log order in the sense of \cref{def:log order}. Hence $(h^*)^{\otimes a}\otimes h^{\otimes b}$ is the Hermitian-Yang-Mills metric  for $T^{a,b}(E,\theta)|_{X-D}$, which is  also adapted to log order. It follows from \cref{thm:criteria} below that $T^{a,b}(E,\theta)$ is also $\mu_L$-polystable. 
	\end{proof}
\subsection{Simpson-Mochizuki correspondence for systems of log Hodge bundles}
A typical and important class of log Higgs bundle is the \emph{system of log Hodge bundles}.  In this subsection, we shall apply \cref{thm:mochizuki} to prove the \emph{Simpson-Mochizuki correspondence for systems of log Hodge bundles}.
\begin{dfn}[System of log Hodge bundles]
	Let $(E,\theta)$ be a log Higgs bundle on a   log pair $(X,D)$. We say that $(E,\theta)$  is a   \emph{system of log Hodge bundles} if there is a decomposition of $E$ into holomorphic vector bundles $E:=\oplus_{p+q=w}E^{p,q}$ such that
	$$
	\theta:E^{p,q}\to E^{p-1,q+1}\otimes \Omega_X^1(\log D).
	$$
	When $D=\varnothing$, such $(E,\theta)$ is called a   \emph{system of Hodge bundles}.  A  system of log Hodge bundles  is $\mu_{\alpha}$-(poly)stable  if it is $\mu_\alpha$-(poly)stable in the sense of log Higgs bundles.
\end{dfn}
\begin{dfn}[Hodge metric]
	Let $(E:=\oplus_{p+q=w}E^{p,q},\theta)$  be a system of Hodge bundles on a complex manifold $X$. A hermitian metric $h$ for $E$ is called a \emph{Hodge metric} if $h$ is harmonic, and it is a direct sum  of metrics on the bundles $E^{p,q}$.
\end{dfn}
By Simpson \cite{Sim88}, a system of Hodge bundles equipped with a Hodge metric is equivalent to a \emph{complex variation of Hodge structures}. 
He then established his correspondence for Hodge bundles over compact K\"ahler manifolds in \cite[Proposition 8.1]{Sim88}.  In the rest of this subsection, we will   extend his result to the log setting. 
  

Let us state and prove the main result in this subsection.
\begin{proposition}\label{prop:SM correspondence}
	Let $(X,D)$ be a projective log pair. Let $(E,\theta)=(\oplus_{p+q=w}E^{p,q},\theta)$ be a system of log Hodge bundles on $(X,D)$ which is $\mu_L$-polystable with $\int_Xc_1(E)\cdot c_1(L)^{\dim X-1}=\int_X ch_2(E)\cdot c_1(L)^{\dim X-2}=0$.   Then there is a  decomposition $(E,\theta)=\oplus_{i\in I}(E_i,\theta_i)$ where each $(E_i,\theta_i)$ is $\mu_L$-stable system of log Hodge bundles so that there is a Hodge metric $h_i$ (unique up to a positive multiplication) for $(E_i|_{X-D},\theta_i|_{X-D})$ which is adapted to  $(E_i,\theta_i)$.
\end{proposition}
\begin{proof}
	Let us first prove the proposition when $(E,\theta)$ is stable.  
	By \cite[Theorem 9.1 \& Propositions 5.1-5.3]{Moc06}, there is a harmonic metrics $h$ for $(E|_{X-D},\theta|_{X-D})$ which is adapted to $(E,\theta)$, and such a harmonic metric is unique up to a positive constant multiplication. We introduce  automorphism $f_t:E\to E$ of $E$ parametrized by $t\in U(1)$, defined by
	\begin{align}\label{eq:Cstar}
f_t(\sum_{p+q=w}e^{p,q})=\sum_{p+q=w}t^pe^{p,q}.
	\end{align}
for every  $e^{p,q}\in E^{p,q}$. Then $f_t:(E,\theta)\to (E,t\theta)$ is an isomorphism since $t\theta\circ f_t=f_t\circ \theta$. Hence by the uniqueness of harmonic metrics, there is a function $\lambda(t):U(1)\to \mathbb{R}^+$  such that
$$
f_t^*h=\lambda(t)\cdot h.
$$
For every  $e^{p,q}\in E^{p,q}$, one has 
$$\lambda(t)\cdot h(e^{p,q},e^{p,q})=f_t^*h(e^{p,q},e^{p,q})=h(f_t(e^{p,q}),f_t(e^{p,q}))=|t^p|^2h(e^{p,q},e^{p,q})=h(e^{p,q},e^{p,q})$$
Hence $\lambda(t)\equiv 1$ for $t\in U(1)$, namely $f_t^*h=h$. On the other hand,
$$
h(e^{p,q},e^{r,s})=f_t^*h(e^{p,q},e^{r,s})=h(f_t(e^{p,q}),f_t(e^{r,s}))=t^pt^{-r}h(e^{p,q},e^{r,s})
$$
for any $t\in U(1)$. Therefore, $h(e^{p,q},e^{r,s})=0$ if $p\neq r$. 
Hence $h$ is a direct sum of hermitian metrics for $E^{p,q}$, namely $h$ is a Hodge metric. The proposition is proved if $(E,\theta)$ is stable.
	 
Let us prove the general cases. 
By \cite[Corollary 3.11 \& Theorem 9.1 \& Propositions 5.1-5.3]{Moc06}, there is a \emph{canonical and unique}  decomposition $(E,\theta)=\oplus_{i\in I}(E_i,\theta_i)\otimes \bC^{p_i}$ where $I$ is a finite set and harmonic metrics $h_i$ for $(E_i|_{X-D},\theta_i|_{X-D})$ which is adapted to $(E_i,\theta_i)$ so that $(E_i,\theta_i)$ is a $\mu_L$-stable log Higgs bundle. By the above arguments, it suffices to prove that each $(E_i,\theta_i)$ is system of log Hodge bundles. Since $(E,\theta)$ is a system of log Hodge bundles,   $(E,t\theta)$ is isomorphic to $(E,\theta)$ for any $t\in U(1)$. We have the following decomposition $(E,t\theta)=\oplus_{i}(E_i,t\theta_i)\otimes \bC^{p_i}$. Note that $(E_i,t\theta_i)$ is still $\mu_L$-stable. By the uniqueness of the decomposition, $(E_i,t\theta_i)\simeq (E_{i_t},\theta_{i_t})$ for some $i_t\in I$.  Since   $I$ is a finite set,  there exists $t_1,t_2$ so that $t_1/t_2$ is not a root of unity and $i_{t_1}=i_{t_2}$. In other words, $(E_i,t_1\theta_i)\simeq (E_i,t_2\theta_i)$. By \cite[Lemma 4.1]{Sim90} or \cite[Theorem 8]{Sim92}, $(E_i,t_1\theta_i)$ is a system of  log Hodge bundles, and so is $(E_i,\theta_i)$.  Hence $(E,\theta)$ is a direct sum of $\mu_L$-stable system of log Hodge bundles $(E_i,\theta_i)$, and each $(E_i|_{X-D},\theta_i|_{X-D})$ admits  a Hodge metric $h_i$ adapted to $(E_i,\theta_i)$. The proposition is proved.
\end{proof}
\section{Principal system of log Hodge bundles}
In this section, we will extend Simpson's \emph{principal system of log Hodge bundles} in \cite[\S 8]{Sim88} to the log setting. We will provide all necessary proofs for the claims for completeness sake. Let us mention that most results in this section follows  from \cite[\S 8 \& \S 9]{Sim88} with   minor changes.

Let $G_0$ be a real   algebraic group which is semi-simple   with its Lie algebra denoted by $\kg_0$. Let $G$ be the complexification of $G_0$ with its Lie algebra denoted by $\kg$. Then $\kg=\kg_0+\sn \kg_0$. $G_0$ is called a \emph{Hodge group} if the following conditions hold.
\begin{itemize}[leftmargin=0.4cm]
	\item The  Lie algebra $\kg$ of $G$ admits a Hodge structure of weight $0$, namely, one has a decomposition
	$$
	\kg=\oplus \kg^{p,-p}
	$$
	so that $[\kg^{p,-p},\kg^{q,-q}]\subset \kg^{p+q,-p-q}$.
	\item  If $\overline{\bullet}$ denotes the complex conjugation with respect to $\kg_0$, then $\overline{\kg^{p,-p}}=\kg^{-p,p}$.
	\item The   form \begin{align}\label{eq:hermitian}
	h_{\kg}(U,V):=(-1)^{p+1}Tr(ad_Uad_{\bar{V}})\quad \mbox{for}\quad U,V\in \kg^{p,-p}
	\end{align} is a positively definite hermitian metric for $\kg$. 
\end{itemize}  
let $K_0\subset G_0$ be the  Lie subgroup of $G_0$ so that its Lie algebra   $\kk_0$ is $\kg_0\cap \kg^{0,0}$. Let $K\subset G$ (resp. $\kk$) be the complexification of $K_0$ (resp. $\kk_0$), and thus the Lie algebra of $K$ is $\kk$. Then the restriction of the Killing form of $\kg_0$ on $\kk_0$ is positively definite, and thus $K_0$ is a compact real Lie group.   

In the rest of the paper, we shall use the above notations without recalling their meanings.

The following concrete example of the Hodge group    will be used in this paper, especially in the proof of \cref{main}.
\begin{example}  \label{ex}
Consider the a direct sum of  $\bC$-vector spaces
$$
V=\oplus_{i+j=w}V^{i,j}
$$
 Denote by $r_i:=\rank\, V^{i,j}$, and $r:=\rank\, V$.  Fix a hermitian metric   $h=\oplus_{i+j=w}h_{i}$  for $V$ where $h_i$ is a hermitian metric for $V^{i,j}$. We take a sesquilinear form $Q(u,v):=(\sn)^{i-j}h(u,v)$ for $u,v\in V^{i,j}$.  Define $G_0:=PU(V,Q)\simeq PU(p_0,q_0)$, where $p_0:=\sum_{i\ odd}r_i$ and $q_0:=\sum_{i\ even}r_i$. We shall show that $G_0$ is a \emph{Hodge group}. 
 
 First we note that the complexification of $G_0$ is $G:=PGL(V)\simeq PGL(r,\bC)$.  
Then the Lie algebra of $G$ is $\kg=\ks\kl(V)\simeq \ks\kl(r,\bC)$, and   the Lie algebra of $G_0$ is $\kg_0=\ks\ku(p_0,q_0)$. Let us define the \emph{Hodge decomposition} as follows:
 $$\kg^{p,-p}=\oplus_{i}\hom(V^{i,j},V^{i+p,j-p})\cap \ks\kl(V).$$
 Then $\kg=\oplus \kg^{p,-p}$.  One can check that $\overline{\kg^{p,-p}}=\kg^{-p,p}$, where the conjugate is taken with respect to the real form $\kg_0$ of $\kg$.
 
 Let $K$ be the subgroup of $G$ which fix each $V^{i,j}$. Then $K=P(\prod_{i+j=w}GL(V^{i,j}) )$, and its Lie algebra is 
 $
 \kk=\kg^{0,0}.
 $ 
 Define $K_0:=K\cap G_0=P(\prod_{i+j=w}U(V^{i,j},h_{i}))$, whose Lie algebra is $\kk_0=\kg^{0,0}\cap \kg_0$.
 
More precisely, if we fix a unitary frame  $e_1,\ldots,e_{p_0}$ for $(\oplus_{i odd}V^{i,j},\oplus_{i\ odd}h_{i})$ and  a unitary frame  $f_1,\ldots,f_{q_0}$ for $(\oplus_{i\ even}V^{i,j},\oplus_{i odd}h_{i})$,  elements in $\kg_0$ can be expressed as  the ones in  $M(r\times r,\bC)$ with the form
  $$
 \begin{bmatrix}
  A & C  \\
  C^* & B  
  \end{bmatrix}
  $$
  where $A\in \ku(p_0)$ and $B\in \ku(q_0)$ so that $Tr(A)+Tr(B)=0$. Note that the Killing form 
  $$
  Tr(ad_uad_{v})=2rTr(uv), 
  $$
  if we consider $u,v$  as elements in $\ks\kl(r,\bC)$. Moreover,  for $u\in \kg^{p,-p}$, one can show that 
  $$
 \overline{u}= \begin{cases}
-u^*\quad \mbox{if}\ p\ \mbox{is even}\\
u^*\quad \mbox{if}\ p\ \mbox{is odd}.
  \end{cases}
$$
where $u^*$ denotes the conjugate transpose of $u$.
  Hence    the hermitian metric $h_\kg$  defined in  \eqref{eq:hermitian} can be simply expressed as 
  $$
h_{\kg}(u,v)= 2rTr(uv^*)
  $$
 once we consider $u,v$  as elements in $\ks\kl(r,\bC)$. 
   In other words, for the natural inclusion $\iota:\kg\hookrightarrow \kg\kl(V)$, one has $h_{\kg}=2r\cdot \iota^*h_{End(V)}$, where $h_{End(V)}$ is the hermitian metric on $\End(V)$ induced by $h_V$. This fact is an important ingredient in the proof of \cref{main}.
   
\end{example}
Let us generalize Simpson's definition of \emph{principal system of  Hodge bundles} in \cite[\S 8]{Sim88} to the log setting as follows.
\begin{dfn}[Principal system of log Hodge bundles]
 A  \emph{principal system of log Hodge bundles} on  a log pair $(X,D)$ is a pair $(P,\tau)$, where $P$ is a holomorphic $K$-fiber bundle    endowed with a holomorphic map
	$$
\tau:T_X(-\log D)\to P\times_K\kg^{-1,1}
	$$
	such that $[\tau(u),\tau(v)]=0$. A \emph{metric}   for $P|_{X-D}$ is a reduction $P_H\subset P|_{X-D}$ whose structure group is $K_0$. Let $d_H$ be the Chern connection for $P_H$. Define $\overline{\tau}_H$ to be the complex conjugate of $\tau|_{X-D}$ with respect to the reduction $P_H$. Then
 $$
	\overline{\tau}_H\in \sC^{\infty}(X-D,(P_H\times_{K_0} \kg^{1,-1})\otimes \Omega_{X-D}^{0,1}).
	$$
 Set
\begin{align}\label{eq:induced flat} D_H:=d_H+\tau|_{X-D}+\overline{\tau}_H,\end{align}
which  is a connection on the smooth   $G_0$-bundle $P_H\times_{K_0} G_0$.  Such triple $(P|_{X-D},\tau|_{X-D},P_H)$ is called  a \emph{principal variation of Hodge structures} over $X-D$ of Hodge group $G_0$, if    the induced connection $D_H$ in \eqref{eq:induced flat}  is \emph{flat}, namely the curvature of $D_H$ is zero. 
\end{dfn}
\begin{rem}\label{rem:well-def}
	Note that the metric reduction $P_H$ for a principal system of  Hodge bundles $(P,\tau)$ on a complex manifold $X$ induces a hermitian metric $h_H$ on $P\times_K\kg\simeq P_H\times_{K_0}\kg$ defined by 
	\begin{align}\label{eq:metric induced}
	h_H\big((p,u),(p,v) \big):=h_{\kg}(u,v)
	\end{align}
	for any $p\in P_H$ and $u,v\in \kg$. Here $h_{\kg}$ is the hermitian metric defined in \eqref{eq:hermitian}.  Note that $K_0$ preserves the decomposition $
	\kg=\oplus_{p+q=w}\kg^{-p,p}$. It thus also preserves  $h_{\kg}$. Indeed,  for $u,v\in \kg^{-p,p}$ and $k\in K_0$, one has
	\begin{align*} 
(-1)^{p+1}h_{\kg}(Ad_ku,Ad_kv) =(-1)^{p+1}h_{\kg}(u,v).
	\end{align*}
	By the equivalence relation $(p,u)\sim (pk^{-1},Ad_{k}u)$, the metric $h_H$ is thus well-defined. 
\end{rem}

\begin{rem}[Period map of principal variation of Hodge structures]\label{rem}
By Simpson \cite[p. 900]{Sim88}, for a principal variation of Hodge structures  $(P,\tau,P_H)$ on a complex manifold $X$, one can also define its \emph{period map} as follows. 	 Denote by $\pi:\tilde{X}\to X$ the universal cover  of $X$. Set $(\tilde{P}:=\pi^*P,\tilde{\tau}:=\pi^*\tau,\tilde{P}_H:=\pi^*P_H)$, which is    a  principal variation of Hodge structures on the simply connected complex manifold $\tilde{X}$. The flat connection $D_H$ thus induces a flat trivialization $\tilde{P}_H\times_{K_0}G_0\simeq \tilde{X}\times G_0$. 
Denote by $\phi:\tilde{P}_H\to G_0$ the composition of the inclusion $\tilde{P}_H\subset \tilde{P}_H\times_{K_0}G_0\simeq \tilde{X}\times G_0$ and the projection $\tilde{X}\times G_0\to G_0$. It induces a map
	\begin{align}\label{eq:period}
		f:	\tilde{X}&\to \faktor{G_0}{K_0}=:\sD\\\nonumber
		\tilde{x}&\mapsto  \phi(e_x)\cdot K_0 \quad \forall e_x\in \tilde{P}_{H,\tilde{x}}.
	\end{align}
Alternatively, we view $G_0\to \sD$ as a principal $K_0$-fiber bundle over $\sD$, and its pull-back on $\tilde{X}$ via $f$ is nothing but the principal $K_0$-fiber bundle $\tilde{P}_H$ by our definition of $f$.
 Hence 
the  complexified   differential of $f$ is 
$$df^\bC:T^{\bC}_{\tilde{X}}\to f^*T^\bC_{\sD}\simeq f^*(G_0\times_{K_0}\oplus_{p\neq 0}\kg^{p,-p})=\tilde{P}_H\times_{K_0}\oplus_{p\neq 0}\kg^{p,-p}$$ 
One can prove that $df^\bC=\tilde{\tau}+\overline{\tilde{\tau}}_H$, where $\overline{\tilde{\tau}}_H$ is the conjugate of $\tilde{\tau}$ with respect to $\tilde{P}_H$. Hence  the restriction of $df^\bC$ to the holomorphic tangent bundle $T_{\tilde{X}}$ is $\tilde{\tau}$, which is a holomorphic map since the holomorphic tangent bundle of $\sD$ is $T_{\sD}\simeq G_0\times_{K_0}\oplus_{p<0}\kg^{p,-p}$. In conclusion, $f$ is a holomorphic map, which is called the \emph{period map}  associated to the principal variation of Hodge structures $(P,\tau,P_H)$, whose differential is given by $df=\tilde{\tau}$.
\end{rem}
The uniformization is  related by Hodge group of Hermitian type.
\begin{dfn}[\!\protecting{\cite[\S 9]{Sim88}}]\label{def:hermitian}
	A Hodge group $G_0$ is called \emph{Hermitian type} if the Hodge decomposition $\kg$ of the Lie algebra of $G$  is
	$$
	\kg=\kg^{-1,1}\oplus \kg^{0,0}\oplus \kg^{1,-1}
	$$
	and that $G_0$ has no compact factor.
	In this case, $K_0\subset G_0$ is the maximal compact subgroup and $\sD:=\faktor{G_0}{K_0}$ is a Hermitian symmetric space of   non-compact type.
\end{dfn}
Let us generalize the definition of \emph{uniformizing bundle} by Simpson \cite[\S 9]{Sim88} to the log setting.
\begin{dfn}[Uniformizing bundle]\label{def:uniformizing}
	Let $G_0$ be a Hodge group of Hermitian type. A \emph{uniformizing bundle} on a log pair $(X,D)$ is a principal system of log Hodge bundles $(P,\tau)$ such that 	 $
	\tau:T_X(-\log D)\stackrel{\simeq}{\to} P\times_K\kg^{-1,1}
	$ is an isomorphism. A \emph{uniformizing variation of Hodge structures} is a uniformizing bundle on a complex manifold $X$ together with a flat metric $P_H\subset P$.
	\end{dfn}

\begin{rem}[Uniformization via uniformizing bundles]\label{rem:complete}
It follows from \cref{def:uniformizing} that, for a uniformizing variation of Hodge structures $(P,\tau,P_H)$ over a complex manifold $X$, the period map $f:\tilde{X}\to \sD$ defined in \eqref{eq:period} is locally biholomorphic. This follows from the fact that $df=\tilde{\tau}$, which is isomorphic at any point of $\tilde{X}$ by the definition.   Recall that in \cref{rem:well-def} the metric reduction $P_H$ together with the positively definite form $h_{\kg}$ for $\kg$ in \eqref{eq:hermitian} induce  a metric $h_H$ for $P\times_K\kg^{-1,1}$.    For the period domain $\sD$ which is a hermitian symmetric space, one can also define the hermitian metric $h_{\sD}$ for $T_{\sD}\simeq G_0\times_{K_0} \kg^{-1,1}$ in a similar way.  By \cref{rem}, 
 $\tilde{P}_H=f^*G_0$ when we consider $G_0\to \sD$ as a principal $K_0$-fiber bundle over $\sD$.  One thus has 
\begin{align} \label{eq:pullback}
\pi^*\tau^*h_H=f^*h_{\sD}.
\end{align} 
 In other words,  $f:(\tilde{X},h_{\tilde{X}}:=\pi^*\tau^*h_H)\to (\sD,h_{\sD})$ is a \emph{local isometry}. Hence for the action of $\pi_1(X)$ on $\tilde{X}$,  the metric $h_{\tilde{X}}$ is invariant under this $\pi_1(X)$-action. If $\tau^*h_H$ is a complete metric, so is $\pi^*\tau^*h_H$, and by \cite[Theorem \rom{4}.1.2]{Cha06}, $f:\tilde{X}\to \sD$ is a Riemannian covering map, which is thus  a biholomorphism since $\tilde{X}$ and $\sD$ are both simply connected. In other words, $X$ is uniformized by the hermitian symmetric space $\sD$ when the metric $\tau^*h_H$ on $X$ is complete.
\end{rem}

One can construct systems of log Hodge bundles from principal ones via Hodge representations.
\begin{dfn}[\!\protecting{\cite[p. 900]{Sim88}}]\label{def:representation}
Let $(V=\oplus_{p+q=w} V^{p,q},h_V)$ be a polarized Hodge structure.	A \emph{Hodge representation} of $G_0$ is a complex representation $\rho:G\to GL(V)$ satisfying the following conditions.
\begin{itemize}[leftmargin=0.4cm]
	\item  The action of $\kg$ is compatible with Hodge type, and such that $K_0$ preserves Hodge type. In other words, 
	$$
	d\rho(\kg^{r,-r})(V^{p,q})\subset V^{p+r,q-r}
	$$
	and
	 $
\rho(K_0)(V^{p,q})\subset V^{p,q}.$\footnote{As remarked by Simpson \cite{Sim88}, this is not automatic if $K_0$ is not connected. However, in   \Cref{ex}, $K_0$ is always connected, and thus such condition will be superfluous in that case.}  
\item	The sesquilinear form  $Q$ defined by
\begin{align} \label{eq:sesquilinear}
Q(u,v):=(\sn)^{p-q}h_V(u,v) \quad \mbox{for} \quad u,v\in V^{p,q}
\end{align} 
	is $G_0$ invariant.  Namely, one has $\rho(G_0)\subset U(V,Q)$. 
\end{itemize}
\end{dfn}
\begin{example}\label{ex:adjoint}
 For the Hodge group $G_0$, $(\kg=\oplus_{p}\kg^{p,-p},h_{\kg})$ is a polarized Hodge structure of weight $0$, where $h_{\kg}$ is the polarization defined in \eqref{eq:hermitian} via the Killing form.  One can easily check that the adjoint representation $Ad:G\to GL(\kg)$ is a Hodge representation for this polarized Hodge structure.  
\end{example}
A principal system of log Hodge bundles together with a Hodge representation induces a system of log Hodge bundles as follows.

\begin{lem}\label{lem:Hodge}
If  $\rho:G\to GL(V)$ is a Hodge representation of the Hodge group $G_0$ and $(P,\tau)$ is  a principal system of log Hodge bundles on the log pair $(X,D)$, then $(E:=P\times_{K}V,\theta:=d\rho(\tau))$ is a system of log Hodge bundles. A polarization $h_V$ for $V$  together with a metric $P_H$ for $P|_{X-D}$ give a metric $h_E$ on the system of Hodge bundles $(E,\theta)|_{X-D}$ over $X-D$. When $(P|_{X-D},\tau|_{X-D},P_H)$ is a principal variation of Hodge structures over $X-D$, $(E|_{X-D},\theta|_{X-D},h_E)$
gives rise to a complex variation of Hodge structures.
\end{lem}
\begin{proof}
By 	\cref{def:representation}, one has $\rho(K)(V^{p,q})\subset V^{p,q}$. Hence $E:=P\times_KV$  admits a   decomposition of holomorphic vector bundles $E=\oplus_{p+q=w}E^{p,q}$ with $E^{p,q}:=P\times_K V^{p,q}$. Let us define  $\theta:=d\rho(\tau) $.
Since $\tau:T_{X}(-\log D)\to P\times_K\kg^{-1,1}$ satisfies $[\tau(u),\tau(v)]=0$, and $d\rho(g^{-1,1})(V^{p,q})\subset V^{p-1,q+1}$,  one thus has $\theta: E^{p,q}\to E^{p-1,q+1}\otimes \Omega_X^1(\log D)$, with $\theta\wedge \theta=0$. Hence $(E,\theta)$ is a system of log Hodge bundles.   
	
Let us now prove that $\rho|_{K_0}:K_0\to GL(V)$ has image on $U(V,h_V)$. Since $\rho(K)(V^{p,q})\subset V^{p,q}$, one thus has
$$
\rho(K)\subset \prod_{p+q=w}GL(V^{p,q}).
$$
Since the  sesquilinear form $Q$ in \eqref{eq:sesquilinear} is $G_0$  invariant, one thus has 
$$
\rho(G_0)=U(V,Q).
$$
Hence
\begin{align}\label{eq:unitary}
\rho(K_0)\subset \rho(G_0\cap K)\subset  \prod_{p+q=w}U(V^{p,q},h_{p,q})\subset U(V,h_V).
\end{align}
Note that $E=P\times_KV\simeq P_H\times_{K_0}V$. We define the hermitian metric $h_E$ for $E$ by setting
\begin{align}\label{eq:fff}
h_E((p,u),(p,v)):=h_V(u,v) 
\end{align}
for any $p\in P_H$ and  for any $u,v\in V$. Since $\rho(K_0)\subset U(V,h_V)$, one can check as \cref{rem:well-def} that $h_E$ is well-defined.

If $(P|_{X-D},\tau|_{X-D},P_H)$ is a principal variation of Hodge structures on $X-D$, the connection $D_H:=d_H+\tau+\overline{\tau}_H$ is flat. By construction, the connection $D_{h_E}:=d_{h_E}+\theta+\overline{\theta}_{h_E}$ for $E|_{X-D}$ is also flat, where $d_{h_E}$ is the Chern connection for the metrized vector bundle $(E,h_E)$, and $\overline{\theta}_{h_E}$ is the conjugate of $\theta$ with respect to $h_E$. Indeed, it can be seen from that $d_{h_E}$ is naturally induced by $d_H$, $\theta:=d\rho(\tau)$, and $\overline{\theta}_{h_E}=d\rho(\overline{\tau}_H)$ by \eqref{eq:fff}. By \cite[p. 898]{Sim88}, the triple $(E|_{X-D},\theta|_{X-D},h_E)$ gives rise to a complex variation of Hodge structures on $X-D$.
	\end{proof}

Conversely, one can associate a system of log Hodge bundles with a principal one as follows. The following result shall be applied in the proof of \cref{main}.
\begin{proposition}\label{prop:converse}
	Let $(E,\theta)=(\oplus_{p+q=w}  E^{p,q},\theta)$ be a system of log Hodge bundles on a log pair $(X,D)$. Then there is a  principal system of log Hodge   bundles $(P,\tau)$  with the structure group $K$ associated to $(E,\theta)$, where $K$ is the semi-simple Lie group in \Cref{ex}. Moreover, any hermitian metric (not necessarily harmonic) $h:=\oplus_{p+q=w}h_p$ for $E|_{X-D}$ gives rise to a metric reduction $P_H$ for $P|_{X-D}$ with the structure group $K_0$ defined in   \Cref{ex}. 
\end{proposition}
\begin{proof}
We shall adopt the same notions as those in \Cref{ex}.  Denote by $r_p:=\rank\, E^{p,q}$, $r:=\sum_{p+q=w}r_p$ and set $\ell_i:=\sum_{p\geq i}r_i$. We consider the following frame bundle $\tilde{P}$.  The fiber of $\tilde{P}$ over a point $x$ is the set of all ordered bases $e_1,\ldots,e_{r}$ (or say  frames) for $E_x$ such that $e_{\ell_p-r_p+1},\ldots,e_{\ell_{p}}$ is a basis for $E^{p,q}_{x}$. The structure group of $\tilde{P}$ is thus $\prod_{p}GL(r_p,\bC)$, which is the subgroup of $GL(r,\bC)$.   $\tilde{P}$ can be equipped with the holomorphic  structure induced by $E$. Consider the homomorphism $f:GL(r,\bC)\to PGL(r,\bC)=:G$, and set $K=P\big(\prod_{p}GL(r_p,\bC)\big)$ to be the image of $\prod_{p}GL(r_p,\bC)$  under  $f$.  Set $P$ to be the holomorphic $K$-fiber bundle   obtained by extending the structure group of $\prod_{p}GL(r_p,\bC)$ using $f$. 
  
  Note that $P\times_K\kg^{-1,1}=\oplus_{i+j=w}\hom(E^{i,j},E^{i-1,j+1})$. Let us  define $\tau:=\theta$. The pair $(P,\tau)$ is a principal system of log Hodge bundles on the log pair $(X,D)$.
	
Recall that the metric $h$ for the Hodge bundle $(E,\theta)|_{X-D}$  is   a direct sum   $h=\oplus_{p+q=w}h_p$. We take a sesquilinear form $Q$ of $E$ defined by  $Q(u,v):=(\sn)^{p-q}h(u,v)$  for $u,v\in E^{p,q}$.  We take $\tilde{P}_H$ to be a reduction of $\tilde{P}|_{X-D}$ consisting of unitary frames with respect to $Q$. In other words, The fiber of $\tilde{P}$ over a point $x$ is the set of frames $e_1,\ldots,e_{r}$  for $E_x$ such that $e_{\ell_p-r_p+1},\ldots,e_{\ell_p}$ is an orthonormal basis for $(E_x^{p,q},h_p)$.  Hence the structure group of $\tilde{P}_H$ is $\tilde{K}_0:=\prod_{p+q=w}U(r_p)$.  Define $K_0:=P\big(\prod_{p+q=w}U(r_p)\big)$, which is the image  $f(\tilde{K}_0)$. Set $P_H$ to be the  smooth principal $K_0$-fiber bundle  on $X-D$ obtained by extending the structure group of $\tilde{P}_H$ using $f:K\to K_0$. Then $P_H\subset P_{X-D}$ is also a metric reduction.   The Hodge group $G_0$ will be  $PU(p_0,q_0)$ where $p_0:=\sum_{p\ even}r_p$ and $q_0:=\sum_{p\ odd}r_p$, and $G:=PGL(r,\bC)$ is the complexification of $G_0$. The proposition is proved.
\end{proof}

\section{Tannakian consideration}
In this section, we shall construct \emph{principal variation of Hodge structures} over quasi-projective manifolds.  Its proof is based on \cref{prop:SM correspondence} together with some Tannakian considerations  in \cite{Sim90,Moc06,Mau15}.  
\begin{thm}\label{Tannakian}
Let $(X,D)$ be a projective log pair endowed with an ample polarization $L$. Let $(P,\tau)$ be a principal system of log Hodge bundles on $(X,D)$, and  let $\rho$ be  a Hodge representation $\rho:G\to GL(V)$ for some polarized Hodge structure $(V=\oplus_{i+j=w}V^{i,j},h_V)$ so that $\rho|_{K_0}: K_0\to GL(V)$ is faithful and $d\rho: \kg_0\to \kg\kl(V)$ is injective.   If the system of log Hodge  bundles $(E:=P\times_K V,\theta:=d\rho(\tau))$ defined in \cref{lem:Hodge} is $\mu_L$-polystable with $\int_X ch_2(E)\cdot c_1(L)^{\dim X-2}=0$, then there exists a metric  reduction $P_H$ for $P|_{X-D}$  so that the triple $(P|_{X-D},\tau|_{X-D},P_H)$ is a \emph{principal variation of Hodge structures} on $X-D$. Moreover, such $P_H$ together with the polarization $h_V$ for $V$ gives rise to a Hodge metric $h$ for $(E,\theta)|_{X-D}$ (defined in \cref{lem:Hodge})  which is adapted to  $(E,\theta)$.
\end{thm}
\begin{proof}
We first prove that $(E,\theta)|_{X-D}$ admits a Hodge metric $h$ over $(E,\theta)|_{X-D}$ which is adapted to   $(E,\theta)$. Since $K$ is  a complex semi-simple Lie group, the Hodge representation $\rho':K\to GL(\det V)$ induced by $\rho$ has image contained in $SL(\det  V)=1$. Hence $\rho'$ is trivial. Note that $\det E=P\times_K \det V$, which is thus a trivial line bundle on $X$. Hence $c_1(E)=0$. Since we assume that  $(E,\theta)$ is $\mu_L$-polystable with $\int_X ch_2(E)\cdot c_1(L)^{\dim X-2}=0$, it follows from \cref{prop:SM correspondence} that $(E,\theta)|_{X-D}$ admits a Hodge metric $h$ over $(E,\theta)|_{X-D}$ which is adapted to    $(E,\theta)$. 

Let us show that $\rho|_{K}:K\to GL(V)$ is  faithful. By \eqref{eq:unitary}, one has $\rho(K_0)\subset U(V,h_V)$. Since $K$ is the complexification of $K_0$ and $\rho|_{K_0}:K_0\to GL(V)$ is assumed to be faithful, one concludes that $\rho|_{K}:K\to GL(V)$ is also faithful.
	
Let us now recall some Tannakian arguments.	The representation $\rho$ induces a representation $\rho_{a,b}:G\to GL(T^{a,b}V)$ for any $a,b\in \mathbb{N}$, where $T^{a,b}V:=\hom(V^{\otimes a},V^{\otimes b})$. 
Since $\rho|_{K}: K\to GL(V)$ is faithful, we can consider $K$ as a reductive algebraic subgroup of $GL(V)$. 
There is a one dimensional complex subspace  $V_1\in T^{a,b}V$ for some $(a,b)\in \mathbb{N}^2$ so that 
\begin{align}\label{eq:tanak}
K=\{g\in GL(V)\mid \rho_{a,b}(g)(V_1)=V_1 \}.
\end{align}
Since $K$ is reductive, there is a complementary subspace $V_2$ of $T^{a,b}V$ for $V_1$ which is invariant under $K$.

By \cref{lem:Hodge}, the Hodge representation $\rho_{a,b}$ and $(P,\tau)$ gives rise to a system of log Hodge bundles $(P\times_K T^{a,b}V, \theta^{a,b}:=d\rho_{a,b}(\tau))$ over $(X,D)$, which is nothing but $T^{a,b}(E,\theta)$. Recall that $\rho_{a,b}(K)(V_1)= V_1$ and $\rho_{a,b}(K)(V_2)=V_2$. Consider the log Higgs bundles $(E_1,\theta_1):=(P\times_K V_1, d\rho_{a,b}(\tau))$ and $(E_2,\theta_2):=(P\times_K V_2, d\rho_{a,b}(\tau))$ over $(X,D)$. 

Note that $T^{a,b}(E,\theta)=(E_1,\theta_1)\oplus (E_2,\theta_2)$. By \cref{thm:mochizuki}, $T^{a,b}(E,\theta)$ is $\mu_L$-polystable with $\int_Xc_1(T^{a,b}(E))\cdot c_1(L)^{\dim X-1}=0$  with respect to an arbitrary polarization $L$. Since $c_1(T^{a,b}(E))=c_1(E_1)+c_1(E_2)$, by the polystability of $T^{a,b}(E,\theta)$, we conclude that $(E_1,\theta_1)$ and $(E_2,\theta_2)$ are both $\mu_L$-polystable. By \cref{prop:SM correspondence}, each $(E_i|_{X-D},\theta_i|_{X-D})$ admits a harmonic metric $h_i$ which is adapted  to $(E_i,\theta_i)$. Moreover, $h$ coincides with $h_1\oplus h_2$ up to some obvious ambiguity.

 In the rest of the proof,   any object which appears is restricted over $X-D$. Let us first  enlarge the structure group of $P$ by defining
$P_{GL(V)}:=P\times_{K}GL(V)$ via the faithful representation $\rho|_{K}:K\to GL(V)$. This is   the holomorphic principal (frame) bundle associated to $E$. We can consider $P=P\times_KK\subset P_{GL(V)}$  as a  reduction of $P_{GL(V)}$.  The metric $h$ for $E$ gives rise to a reduction $P_{U(E,h)}$ of $P_{GL(V)}$  
with the structure group $U(V,h_V)$.   Indeed, note that 
$$
E=P_{GL(V)}\times_{GL(V)}V
$$
and thus the metric $h$ for $E$ induces a family of hermitian metrics $h_e$ for $V$ parametrized by $e\in P_{GL(V)}$. It has the obvious relation $h_{e\cdot g}=g^*h_{e}$ for any $g\in GL(V)$. We define 
\begin{align}\label{eq:reduction}
P_{U(E,h)}:=\{e\in P_{GL(V)}\mid h_e=h_V\}
\end{align} 
and it is obvious that  if $e\in P_{U(E,h)}$, then $e\cdot g\in P_{U(E,h)}$ if and only if $g\in U(V,h_V)$.  Hence the structure group of $P_{U(E,h)}$ is $U(V,h_V)$.

 Let us define $P_H:=P\cap P_{U(E,h)}$ whose structure group is $U(V,h_V)\cap K\supset K_0$ by \eqref{eq:unitary}. Since $K_0$ is the maximal compact subgroup of $K$ and $U(V,h_V)\cap K$ is also compact, one has moreover $U(V,h_V)\cap K=K_0$. Hence $P_H\subset P$ is a metric reduction with the structure group $K_0$.  
 Obviously, if we follow \cref{lem:Hodge}  to define a new metric $h'$ for $E$    by setting
 $$
 h'((p,u),(p,v)):=h_V(u,v)
 $$
 for any $p\in P_H$ and  for any $u,v\in V$, then 
 \begin{align}\label{eq:compatible}
 h'=h
 \end{align} by \eqref{eq:reduction}.  We shall prove that $(P|_{X-D},\tau|_{X-D},P_H)$ is a principal variation of Hodge structures on $X-D$ following the elegant arguments in \cite[Proposition 3.7]{Mau15}.

Let $A\in \sC^\infty(P_{GL(V)},T^*_{P_{GL(V)}}\otimes \kg\kl(V))$ be the Chern connection 1-form for the principal bundle $P_{GL(V)}$     induced by   the Chern connection $d_h$ for $(E,h)$. 
Fix a base point $p\in P\subset P_{GL(V)}$, and we denote by  $\pi:P\to X$  the projection map.  Recall that $$T^{a,b}(E,h)=(E_1,h_1)\oplus (E_2,h_2),$$
and
$$E_i=P\times_K V_i.$$  
Hence the holonomy $Hol(p,\gamma)\in GL(V)$ with respect to the connection $A$ along any smooth loop $\gamma$ based at $\pi(p)$ satisfies that
$$
\rho_{a,b}\big(Hol(p,\gamma)\big)(V_i)\subset V_i
$$ 
for $i=1,2$. By \eqref{eq:tanak}, one has $Hol(p,\gamma)\in K$.  Hence the restriction of $A$ to $P$ is 1-form with values in $\kk$. In other words, $A$ is induced by a connection on  $P$.

On the other hand, by the definition of the Chern connection, $A$ is also induced by a connection on  $P_{U(E,h)}$; in other words, the restriction of $A$ to $P_{U(E,h)}$ is 1-form with values in ${\rm Lie}(U(V,h_V))$, where  ${\rm Lie}(U(V,h_V))$ denotes the Lie algebra of $U(V,h_V)$. Since $\kk_0=\kk\cap {\rm Lie}(U(V,h_V))$, there is a connection $A_0\in \sC^\infty(P_{H},T^*_{P_H}\otimes \kk_0)$ for the smooth principal $K_0$-fiber bundle $P_H:=P_{U(E,h)}\cap P$ which induces the connection $A$. $A_0$  is moreover the Chern connection with respect to the reduction $P_H$ of $P$ by our construction. Let us define $F_H\in \sA^{1,1}(P_H\times_{K_0}\kg_0)$ to be the curvature form of the connection $A_0+\tau+\overline{\tau}_H$ over the smooth principal $K_0$-bundle $P_H\times_{K_0} G_0$, , where $\overline{\tau}_H$ is the adjoint of $\tau$ with respect to the metric reduction $P_H\subset P$. Recall that  
 $
\theta:=d\rho(\tau).
$ 
By \eqref{eq:compatible}, one has $\overline{\theta}_h=d\rho(\overline{\tau}_H)$. Hence 
\begin{align}\label{eq:flat}
d\rho(F_H)=(d_h+\theta+\overline{\theta}_h)^2=F_h(E)=0 
\end{align}
where $d_h$ is the Chern connection for $(E,h)$.
Since $d\rho:\kg_0\to \kg\kl(V)$ is assumed to be injective, by \eqref{eq:flat} this implies that $F_H=0$. In conclusion, $(P|_{X-D},\tau|_{X-D},P_H)$ is a principal variation of Hodge structures on $X-D$. 
	\end{proof}
\section{Uniformization of quasi-projective manifolds by unit balls} 
This section is devoted to the proof of \cref{main}.  In \cref{sec:pluri}  we shall prove a basic result for the extension of plurisubharmonic functions. This lemma will be used in the proof of \cref{main}. We shall also give an application of this fact in Hodge theory: we can give a much simpler proof  of the negativity of kernel of Higgs fields for tame harmonic bundles originally proven by Brunebarbe \cite{Bru17} (see also \cite{Zuo00} for systems of log Hodge bundles).  With all the tools developed above, we are able to prove \cref{main} in \cref{sec:uniform}.
\subsection{Adaptedness  to log order and acceptable metrics}\label{sec:adapt}
We recall some notions in \cite[\S 2.2.2]{Moc07}.  
Let $X$ be a $\mathscr{C}^{\infty}$-manifold, and $E$ be a $\mathscr{C}^{\infty}$-vector bundle with a hermitian metric $h$. Let $\mathbf{v}=(v_1,\ldots,v_r)$ be a $\mathscr{C}^{\infty}$-frame of E. We obtain the $H(r)$-valued function $H(h,\mathbf{v})$,whose $(i,j)$-component is given by $h(v_i,v_j)$.  

Let us consider the case $X=\Delta^n$, and $D=\sum_{i=1}^{\ell}D_i$ with $D_i=(z_i=0)$. We have the coordinate $(z_1,\ldots,z_n)$. Let  $h$, $E$ and $\mathbf{v}$  be as above.

A frame $\mathbf{v}$ is called \emph{adapted up to log order}, if  the following inequalities hold over $X-D$
$$
C^{-1}(-\sum_{i=1}^{\ell}\log |z_i|)^{-M}\leq H(h,\mathbf{v})\leq  C(-\sum_{i=1}^{\ell}\log |z_i|)^{M}
$$   
for some positive numbers $M$ and $C$.
\begin{dfn}\label{def:log order}
	Let $(X,D)$ be a log pair, and let $E$ be a holomorphic vector bundle on $X$. A   hermitian metric $h$ for $E|_{X-D}$ is   \emph{adapted   to log order}  if for any point $x\in D$, there is an admissible coordinate $(U;z_1,\ldots,z_n)$, a holomorphic frame $\mathbf{v}$ for $E|_{U}$ which is  adapted up to log order.
\end{dfn} 

\begin{dfn}[Acceptable metric]\label{def:acceptable}
	Let $(X,D)$ be a log pair and	let $(E_,\theta)$  be a log Higgs bundle over   $(X,D)$. We say that the metric $h$ for $E|_{X-D}$ is \emph{acceptable}, if for any $p\in D$ there is an  admissible coordinate $(U;z_1,\ldots,z_n)$ around $p$, so that the norm  $\lvert F_h \rvert_{h,\omega_P}\leq C$ for some $C>0$ over $U-D$.    Such triple $(E,\theta,h)$ is called an \emph{acceptable bundle} on $(X,D)$. 
\end{dfn}

One can easily check that acceptable metrics and adaptedness to log order defined above  are invariant under bimeromorphic transformations. 
\begin{lem}\label{lem:birational}
	Let $(X,D)$ be a   log pair, and let $\mu:\tilde{X}\to X$ be a   bimeromorphic morphism so that $\mu^{-1}(D)=\tilde{D}$. For a log Higgs bundle $(E,\theta)$  over $(X,D)$, one can define a log Higgs bundle $(\tilde{E},\tilde{\theta})$ on $(\tilde{X},\tilde{D})$ by setting $\tilde{E}=\mu^*E$ and $\tilde{\theta}$ to be the composition
	$$
	\mu^*E\xrightarrow{\mu^*\theta}\mu^*(E\otimes \Omega_X^1(\log D))\to \mu^*E\otimes \Omega^1_{\tilde{X}}(\log \tilde{D}).
	$$ 
	If the metric $h$ for $(E,\theta)|_{X-D}$ is acceptable or adapt to log order, so is the metric $\mu^*h$ for $(\tilde{E},\tilde{\theta})|_{\tilde{X}-\tilde{D}}$. \qed
\end{lem}
 
\subsection{Extension of psh functions and negativity of kernel of Higgs fields} \label{sec:pluri}
In this subsection we shall prove a result on the extension of plurisubharmonic (psh for short) functions, which will be used in the proof of  \cref{main,prop:analysis}. As a byproduct, we give a very simple proof of the negativity of kernels of Higgs fields of tame harmonic bundles by Brunebarbe \cite[Theorem 1.3]{Bru17}, which generalizes the earlier work by Zuo \cite{Zuo00} for system of log Hodge bundles.
\begin{lem}\label{lem:extension2} 
	Let   $X=\Delta^n$, and $D=\sum_{i=1}^{\ell}D_i$ with $D_i=(z_i=0)$. Let $\varphi$ be a psh function  on $X^*$. We assume that for any $\delta>0$, there is a positive constant $C_\delta$ so that 
	$$
	\varphi(z)\leq \delta\sum_{j=1}^{\ell}(-\log |z_j|^2))+C_\delta
	$$
	on $X^*$. Then $\varphi$ extends \emph{uniquely} to a  psh function on $X$.
\end{lem}
\begin{proof}
	Define $\varphi_\ep:=\varphi+\ep \sum_{j=1}^{\ell}(\log |z_j|^2)$ for any $\ep>0$. Then for each $\ep>0$,  $\varphi_\ep$ is locally bounded from above, which thus extends to a  psh $\tilde{\varphi}_\ep$ on the whole $X$ by the well-known fact in pluripotential theory.  By the maximum principle, for any $0<r<1$,  there is a point $\xi_\ep\in  S(0,r)\times\cdots S(0,r)$ so that 
	$$
	\sup_{z\in \Delta(0,r)\times \cdots\times \Delta(0,r)}\varphi_\ep(z)\leq \varphi_\ep(\xi_\ep)\leq \varphi(\xi_\ep)
	$$
	where $S(0,r):=\{z\in \Delta\mid |z|=r\}$.  Note that the compact set  $S(0,r)\times\cdots S(0,r)$ is contained in $X-D$. Since $\varphi$ is psh on $X-D$, there exists $z_0\in S(0,r)\times\cdots S(0,r)$ so that
	$$
	\sup_{z\in S(0,r)\times\cdots S(0,r)}\varphi(z)\leq \varphi(z_0)<+\infty. 
	$$
	Hence $\varphi_\ep$ is \emph{uniformly} locally bounded from above.

	We define  the \emph{upper envelope }
 $
	\tilde{\varphi}:=	\sup_{\ep>0} \tilde{\varphi}_\ep$,  
	and define the  \emph{upper semicontinuous regularization} of $\tilde{\varphi}$ by 
 $
	\tilde{\varphi}^\star(x):=\lim_{\delta\to 0^+}\sup_{\mathbb{B}(x,\delta)}	\tilde{\varphi}(z)$,  
	where $\mathbb{B}(x,\delta)$ is the unit ball of radius $\delta$
	centered at $x$. Then by the well-known result in pluripotential theory \cite[Chapter 1, Theorem 5.7]{Dembook}, $	\tilde{\varphi}^\star$	 is a psh function on $X$. By our construction, $	\tilde{\varphi}^\star(z)=\varphi(z)$ on $X-D$. This proves our result.
\end{proof}
A direct consequence of the above lemma is the following extension theorem of positive currents.
\begin{lem}\label{lem:current}
	Let $(X,D)$ be a log pair and let  $L$ be a line bundle on $X$. Assume that $h$ is a smooth hermitian metric for $L|_{X-D}$, which is adapted to log order. Assume further that the curvature form $\sn R_{h}(L|_{X-D})\geq 0$. Then $h$ extends to a singular hermitian metric $\tilde{h}$ for $L$ with zero Lelong numbers so that the   curvature current $\sn R_{\tilde{h}}(L)$ is closed and positive.  In particular, $L$ is a nef line bundle.\qed
\end{lem}

Let us show how to apply \cref{lem:extension2} to reprove the negativity of kernels of Higgs fields of tame harmonic bundles.
\begin{thm}[Brunebarbe]\label{thm:negativity}
	Let $X$ be a compact K\"ahler manifold and let $D$ be a simple normal crossing divisor on $X$. Let $(E,\theta,h)$ be  a tame harmonic bundle on $X-D$, and let $(\diae,\theta)$ be the prolongation defined in \cite[\S 4.1]{Moc02}. Let $\cF$ be any coherent torsion free subsheaf of $\diae$ which lies in the kernel   of the Higgs field $\theta:\diae\to \diae\otimes \Omega_X^1(\log D)$, namely $\theta(\cF)=0$. Then  
	\begin{thmlist}
		\item the singular hermitian metric $h|_{\cF}$ for $\cF$, is \emph{semi-negatively curved} in the sense of \cite[Definition 2.4.1]{PT18}. 
		\item The dual $\cF^*$ of $\cF$ is \emph{weakly positive over $X^\circ-D$} in the sense of Viehweg, where $X^\circ\subset X$ is the Zariski open set so that $\cF|_{X^\circ}\to \diae|_{X^\circ}$ is a subbundle. 
		\item \label{FF} If the harmonic metric $h$ is adapted to log order and $\cF$ is a subbundle of $\diae$ so that $\theta(\cF)=0$, then the line bundle $\cO_{\bP(\cF^*)}(1)$ admits a singular hermitian metric $g$ with zero Lelong numbers so that the curvature current $\sn R_g(\cO_{\bP(\cF^*)}(1))\geq 0$; in particular, $\cF^*$ is a nef vector bundle.
		\end{thmlist}
\end{thm}
\begin{proof}
	By \cite[Definition 2.4.1]{PT18}, it suffices to prove that for any open set $U$ and any $s\in \cF(U)$,   $\log|s|_h^2$ extends to a psh function on $U$. Pick any point $x\in D$. By the definition of $\diae$, for any $\delta>0$,  there are an admissible coordinate $(U;z_1,\ldots,z_n)$ centered at $x$, and a positive constant $C_\delta$ so that 
	$$
	\log |s|_h^2\leq \delta\sum_{j=1}^{\ell}(-\log |z_j|^2))+C_\delta
	$$
	on $U-D$.  
	Recall that $R_h(E)+[\theta,\overline{\theta}_h]=F_h(E)=0$. Since $\theta(s)=0$,  we have
	\begin{align*}
	\hess \log |s|^2_{h}&\geq-\frac{\sn\{\theta s,\theta s\}}{|s|^2_{h}}-\frac{\sn\{\overline{\theta}_{h}s,\overline{\theta}_{h}s\}}{|s|^2_{h}}\\
	&=-\frac{\sn\{\overline{\theta}_{h}s,\overline{\theta}_{h}s\}}{|s|^2_{h}}\geq 0.
	\end{align*}
	over $X-D$. Hence $	\log |s|_h^2$ is a psh function on $X-D$. 
	By \cref{lem:extension2}, we conclude that $\log |s|_h^2$ extends to a psh function on $U$. This proves that $(\cF,h)$ is negatively curved in the sense of  P\u{a}un-Takayama. 
	
	The metric $h$ induces a negatively curved singular hermitian metric $h_1$ (in the sense of \cite[Definition 2.2.1]{PT18}) on the subbundle $\cF|_{X^\circ}$.  By \cref{lem:current}, $h_1$ induces a singular metric $g$ for the line bundle $\cO_{\bP(\cF^*|_{X^\circ})}(1)$  so that $\sn R_{g}(\cO_{\bP(\cF^*|_{X^\circ})}(1))\geq 0$.  Note that $X-X^\circ$ is a codimension at least two subvariety. The second statement then follows from    H\"ormander's $L^2$-techniques in \cite[Proof of Theorem 2.5.2]{PT18}.  
	
	Let us prove the last statement. Since $\cF$ is a subbundle of $\diae$, one has  $X^\circ=X$. Since $h$ is assumed to be adapted to log order, the singular hermitian metric $g$ for $\cO_{\bP(\cF^*)}(1)$ thus has zero Lelong numbers everywhere. This implies  the nefness of the vector bundle $\cF^*$.
\end{proof}
 
\subsection{Characterization of non-compact ball quotient}\label{sec:uniform}
Let us state and prove our first main theorem in this paper.
\begin{thm}\label{thm:equality}
	Let $X$ be an $n$-dimensional complex projective manifold and let $D$ be a simple normal crossing divisor   on $X$. Let $L$ be an ample polarization on $X$. For the log Hodge bundle
	$(\Omega_X^1(\log D)\oplus \cO_X,\theta)$ on $(X,D)$ with $\theta$  defined in \eqref{eq:Higgs}, 
	we assume that it is $\mu_L$-polystable. Then one has the following inequality
	\begin{align}\label{eq:BM2}
	\big(2c_2(\Omega_X^1(\log D))-\frac{n}{n+1}c_1(\Omega_X^1(\log D))^2\big)\cdot c_1(L)^{n-2}\geq 0. 
	\end{align}
When the above equality holds, 
	\begin{thmlist}
		\item \label{item1} if $D$ is smooth, then  ${X-D}\simeq \faktor{\bB^n}{\Gamma}$ for some  torsion  free lattice  $\Gamma\subset PU(n,1)$ acting on $\bB^n$. Moreover, $X$ is the (unique)   toroidal compactification of $\faktor{\bB^n}{\Gamma}$, and each connected component of $D$ is the \emph{smooth}  quotient of an Abelian variety $A$  by a finite group acting  freely on $A$.
	\item \label{item2}    If   $D$ is not smooth, then the universal cover $\widetilde{X-D}$ of \space $X-D$ is \emph{not} biholomorphic to $\bB^n$, though there exists a holomorphic map $\widetilde{X-D}\to \bB^n$ which is locally biholomorphic.
\end{thmlist} 
In both cases, $K_X+D$ is big, nef and ample over $X-D$.
\end{thm}

\begin{proof} 
Denote  the log Hodge bundle $(E,\theta)=(E^{1,0}\oplus E^{0,1},\theta)$ by
$$
E^{1,0}:=\Omega^1_X(\log D), \quad E^{0,1}:=\cO_X.
$$ 
By \cite[Theorem 6.5]{Moc06} we have the following Bogomolov-Gieseker inequality for $(E,\theta)$
	\begin{align}\label{eq:BG}
	\big(2c_2(\Omega^1_X(\log X))-\frac{n}{n+1}c_1(\Omega^1_X(\log D))^2\big)\cdot c_1(L)^{n-2}=\\\nonumber
	\big(2c_2(E)-\frac{\rank\, E-1}{\rank\, E}c_1(E)^2\big)\cdot c_1(L)^{n-2}\geq 0 
	\end{align}  
	This shows the desired inequality \eqref{eq:BM2}.
	
	The rest of the proof will be divided into three steps. In Step 1, we shall  construct a uniformizing variation of Hodge structures on $X-D$ so that the corresponding period map defined in \eqref{eq:period} induces a holomorphic map (so-called \emph{period map} in \cref{rem:complete}) from the universal cover of $X-D$ to $\bB^n$ which is locally biholomorphic. By \eqref{eq:pullback}, this period map is moreover an \emph{isometry} if we equip   $X-D$ with hermitian metric induced by the Hodge metric.   In Step two we will prove that, when $D$ is smooth,   the hermitian metric on $X-D$ induced by the Hodge metric  is \emph{complete}. Together with arguments in \cref{rem:complete}, this proves that the above period map  is indeed a biholomorphism.   In Step three we shall prove \cref{item2} and the positivity of $K_X+D$.
	
\noindent \textbf{Step 1}.	 We apply  \cref{prop:converse} to the above system of log Hodge bundles $(E^{1,0}\oplus E^{0,1},\theta)$. Then there is a  principal system of log Hodge bundles  $(P,\tau)$ on $(X,D)$ with the structure group $K=P(GL(V^{1,0})\times GL(V^{0,1}))$ with $\rank\, V^{1,0}=\rank\, E^{1,0}=n$, and $\rank\, V^{0,1}=\rank\, E^{0,1}=1$.  Here we use the notations in \Cref{ex}.  Then by  \cref{prop:converse} the Hodge group relative to $(P,\tau)$  is $G_0=PU(n,1)$, and $K_0=K\cap G_0=P(U(n)\times U(1))=U(n)$.  For the complexified group $G=PGL(V)$ of $G_0$, its adjoint  representation $Ad:G\to GL(\kg)=GL(\ks\kl(V))$ is faithful. By \Cref{ex:adjoint}, this is a Hodge representation.  
By \Cref{lem:Hodge}, such Hodge representation $Ad$ induces a system of log Hodge bundles  $(P\times_{Ad}\kg,d(Ad)(\tau))$ over $(X,D)$.  It follows our construction of $(P,\tau)$ that
$$
(P\times_{Ad}\kg,d(Ad)(\tau))=(\End(E)^{\perp},\theta_{End(E)^{\perp}}).
$$
where $\End(E)^\perp$ is the trace-free subbundle of $\End(E)$, and $\theta_{End(E)^{\perp}}$ is the induced Higgs field from $(E,\theta)$. 
	
On the other hand,  an easy computation shows that $c_1(\End(E))=0$, and  
\begin{align*}
ch_2(\End(E))&=-2\rank\, E\cdot c_2(E)+(\rank\, E-1)c_1(E)^2\\
&=nc_1^2(K_X+D)-2(n+1)c_2(\Omega_X^1(\log D))=0 
\end{align*}
since  the equality in \eqref{eq:BG} holds by our assumption. Since we assume that $(E,\theta)$ is $\mu_L$-polystable, by \cref{prop:stability},    $(\End(E),\theta_{End(E)})$ is also $\mu_L$-polystable. We now apply \cref{prop:SM correspondence} to find a Hodge metric $h$ for the system of log Hodge bundle $(\End(E)|_{X-D},\theta_{End(E)}|_{X-D})$ which is adapted to  $(\End(E),\theta_{End(E)})$. Since  $(\End(E),\theta_{End(E)})=(\End(E)^{\perp},\theta_{End(E)^{\perp}}) \oplus (\cO_X,0)$, we conclude that $h=h_1\oplus h_2$, where $h_1$ is the harmonic metric for $(\End(E)^{\perp}|_{X-D},\theta_{End^{\perp}(E)}|_{X-D})$ which is adapted to  the log Higgs bundle $(\End^{\perp}(E),\theta_{End^{\perp}(E)})$, and $h_2$ is the canonical metric for the trivial Higgs bundle $(\cO_X,0)$.

We now apply \cref{Tannakian} to conclude that   $h_1$ induces a  reduction $P_H$ for $P|_{X-D}$ with the structure group $K_0=P(U(n)\times U(1))\simeq U(n)$, which is compatible with $h_1$ such that $(P|_{X-D},\tau|_{X-D},P_H)$ is a principal variation of Hodge structures on $X-D$. Note that
$$
T_X(-\log D)\xrightarrow{\tau}P\times_K\kg^{-1,1}=\hom(E^{1,0},E^{0,1})\simeq \hom(\Omega_X^1(\log D),\cO_X)
$$ 
is an isomorphism. Hence  $(P|_{X-D},\tau|_{X-D},P_H)$ is moreover a \emph{uniformizing variation of Hodge structures} over $X-D$ in the sense of \cref{def:uniformizing}.  By \cref{rem:complete}, it gives rise to a holomorphic map, the  so-called period map, 
\begin{align}\label{eq:period2}
\widetilde{X-D}\to \faktor{G_0}{K_0}=\faktor{PU(n,1)}{U(n)}\simeq \bB^n
\end{align}
 defined in \eqref{eq:period}, which is locally \emph{biholomorphic}.  Here $\widetilde{X-D}$ is the \emph{universal cover} of $X-D$. 
 
  Note that the reduction $P_H$ together with the hermitian metric $h_{\kg}$ in \eqref{eq:hermitian} gives rise to a natural metric $h_{H}$ over $P\times_{K}\kg|_{X-D}$ defined in \eqref{eq:metric induced}. By \cref{rem:complete} again, if the pull back $\tau^*h_H$ is a \emph{complete metric} on $X-D$, then $X-D$ is uniformized by $\faktor{G_0}{K_0}=\faktor{PU(n,1)}{U(n)}$ which is the complex unit ball of dimension $n$,  denoted by $\mathbb{B}^n$. It follows from \eqref{eq:compatible} that   $h_1=h_H$.  It now suffices to  show that $\tau^*h$ is complete if we want to prove that $X-D$ is uniformized by $\bB^n$, where  we recall $$\tau:T_X(-\log D)\stackrel{\simeq}{\to} \hom(E^{1,0},E^{0,1})\subset \End(E).$$
In next step, we will apply similar ideas by Simpson   \cite[Corollary 4.2]{Sim90} to prove this.  Note that until now we made no assumption on the smoothness of $D$.
\medskip

\noindent \textbf{Step 2}. Throughout Step 2, we will assume that $D$ is smooth.  Consider now the system of log Hodge bundles $(\cE,\eta):=(\End(E),\theta_{End(E)})$. We first mention that the above Hodge metric   $h$ for  $(\cE,\eta)|_{X-D}$ is adapted to log order in the sense of \cref{def:log order}. Indeed, it follows from \cite[Corollary 4.9]{Moc02} that the eigenvalues of monodromies of the flat connection $D:=d_h+\eta+\overline{\eta}_h$ around the divisor $D$ are 1.  By the \enquote{weak} norm estimate in \cite[Lemma 4.15]{Moc02}, we conclude that $h$ is adapted to  log order\footnote{Indeed, a strong norm estimate has already been obtained by Cattani-Kaplan-Schmid in \cite{CKS86}. Here we only need to know that $h$ is adapted to log order, which is a bit easier to obtain using Andreotti-Vesentini type results by Simpson \cite{Sim90} and Mochizuki \cite[Lemma 4.15]{Moc02}.}.  

 We  first give an estimate for $\tau^*h$.  For any point $x\in D$, consider an admissible coordinates $(U;z_1,\ldots,z_n)$ centered at $x$ as \cref{def:admissible} so that $D\cap U=(z_1=0)$. 
  To distinguish the sections of log Higgs bundles and log forms, we write 
$e_1:=d\log z_1$ and $e_i=dz_i$ for $i=2,\ldots,n$. 
Denote by $e_0=1$ the constant section of $\cO_X$. 
Let us  introduce a new metric $\tilde{h}$ on $(E,\theta)|_{U^*}$ as follows.
\begin{align*}
|e_1|^2_{\tilde{h}}&:=(-\log |z_1|^2);\quad 
\langle e_i,e_j\rangle_{\tilde{h}}:= 0\quad \mbox{for}\quad i\neq j;\\
|e_i|^2_{\tilde{h}}&:=   1\quad \mbox{for}\quad i=2,\ldots,n;\quad
|e_0|^2_{\tilde{h}}:=  (-\log |z_1|^2)^{-1}.
\end{align*}  
Write $h_{ii}:=|e_i|_{\tilde{h}}^2$,
and $F_{\tilde{h}}(E):=F_{\tilde{h}}(E)_{kj}\otimes e^*_j\otimes e_k$. Then for $i,j=2,\ldots,n$, one has 
\begin{align*}
F_{\tilde{h}}(E)_{11}&=F_{\tilde{h}}(E)_{10}=F_{\tilde{h}}(E)_{01}=F_{\tilde{h}}(E)_{0i}=F_{\tilde{h}}(E)_{j0}=0 \\
F_{\tilde{h}}(E)_{ij}&=(-\log |z_1|^2)^{-1}d\bar{z}_i\wedge d{z}_j \\
F_{\tilde{h}}(E)_{1i}&=\frac{1}{(-\log |z_1|^2)^2\bar{z}_1}d\bar{z}_1\wedge d{z}_i\\
F_{\tilde{h}}(E)_{i1}&=\frac{1}{(-\log |z_1|^2)z_1}d\bar{z}_i\wedge  {d{z}_1}  \\
F_{\tilde{h}}(E)_{00}&=\sum_{i=2}^{n} (-\log |z_1|^2)^{-1}dz_i\wedge d\bar{z}_i.
\end{align*}   
 In conclusion,  there is a constant $C_1>0$ so that one has
\begin{align}\label{eq:bound0} 
|F_{\tilde{h}}(E)|_{h,\omega_e}^2=\sum_{0\leq j,k\leq n}|F_{\tilde{h}}(E)_{kj}\otimes e^*_j\otimes e_k|_{h,\omega_e}^2 \leq \frac{C_1}{(-\log |z_1|^2)^3|z_1|^2}
\end{align}
over $U^*(\frac{1}{2})$ (notation  defined in \cref{def:admissible}), 
where $\omega_e=\sn \sum_{i=1}^{n}dz_i\wedge d\bar{z}_i$ is  the Euclidean metric on $U^*$. 

We abusively denote by $\tilde{h}$ the induced metric on $(\cE,\eta)|_{U^*}:=(\End(E),\theta_{End(E)})|_{U^*}$, which is adapted to log order on $(U,D\cap U)$ in the sense of \cref{def:log order} by our construction. Then 
\begin{align*}
  F_{\tilde{h}}(\cE)&=  F_{\tilde{h}}(E)\otimes \vvmathbb{1}_{E^*}+   \vvmathbb{1}_{E}\otimes F_{\tilde{h}^*}(E^*)\\
&= F_{\tilde{h}}(E)\otimes \vvmathbb{1}_{E^*}-   \vvmathbb{1}_{E}\otimes F_{\tilde{h}}(E)^\dagger 
\end{align*} 
where $F_{\tilde{h}}(E)^\dagger$ is the transpose of $F_{\tilde{h}}(E)$. Hence 
$$
F_{\tilde{h}}(\cE)(e_i\otimes e_j^*)=\sum_{k,\ell}(\delta_{j\ell}F_{\tilde{h}}(E)_{ik} -\delta_{ik}F_{\tilde{h}}(E)_{\ell j})(e_k\otimes e_\ell^*)
$$
for $0\leq i,j,k,\ell\leq n$.
It then follows from \eqref{eq:bound0} that 
\begin{align}\label{eq:bound2}
|F_{\tilde{h}}(\cE)|_{h,\omega_e}^2\leq \frac{C_2}{(-\log |z_1|^2)^3|z_1|^2} 
\end{align}
over $U^*(\frac{1}{2})$ for some constant $C_2>0$.
Consider the identity map $s$ for $\cE$, which can be seen as a holomorphic section of $\End(\cE,\cE)$. We denote by $(\cF,\Phi):=(\End(\cE,\cE),\eta_{End(\cE)})$   the induced Higgs bundle by $(\cE,\eta)$. 
One can check that 
\begin{align}\label{eq:kernel}
\Phi(s)=0.
\end{align} 
We equip $\cF|_{U^*}$ with the  metric $h_\cF:=\tilde{h}\otimes h^{*}$, where $h$ is the harmonic metric for $(\cE,\eta)|_{X-D}$ constructed in Step one. Note that
\begin{align*}
  F_{h_\cF}(\cF)&=  F_{\tilde{h}}(\cE)\otimes \vvmathbb{1}_{\cE^*}+  \vvmathbb{1}_{\cE}\otimes F_{h^*}(\cE^*)\\
&= F_{\tilde{h}}(\cE)\otimes \vvmathbb{1}_{\cE^*}
\end{align*} 
By \eqref{eq:bound2}, 
there is a constant $C_0>0$ so that one has
\begin{align}\label{eq:bound}
|F_{h_\cF}(\cF)|_{h_\cF,\omega_e}\leq \frac{C_0}{(-\log |z_1|^2)^{\frac{3}{2}}|z_1|} 
\end{align}
over $U^*(\frac{1}{2})$.
Then
\begin{align*}
\hess \log |s|^2_{h_\cF}
&\geq -\frac{\sn\{R_{h_\cF}s,s\}}{|s|^2_{h_\cF}}\\
&=-\frac{\sn\{\Phi s,\Phi s\}}{|s|^2_{h_\cF}}-\frac{\sn\{\overline{\Phi}_{h_\cF}s,\overline{\Phi}_{h_\cF}s\}}{|s|^2_{h_\cF}} -\frac{\sn\{F_{h_\cF}(\cF)s,s\}}{|s|^2_{h_\cF}}\\
&\geq  -\frac{\sn\{F_{h_\cF}(\cF)s,s\}}{|s|^2_{h_\cF}}.
\end{align*}
Here the third  inequality follows from \eqref{eq:kernel}. For any $\xi=(\xi_2,\ldots,\xi_n)$ with $0\leq \xi_2,\ldots, \xi_n\leq \frac{1}{2}$, we define  a smooth function $f_\xi$ over $\Delta^*$ parametrized by $\xi$ by
$$
f_\xi(z_1):=\log |s|^2_{h_\cF}(z_1,\xi_2,\ldots,\xi_n).
$$
Then the above inequality together with \eqref{eq:bound} implies that
$$
\Delta f_\xi\geq -|F_{h_\cF}(\cF)|_{h_\cF,\omega_e}\geq -\frac{C_0}{(-\log |z_1|^2)^{\frac{3}{2}}|z_1|}=:\varphi
$$
where $C_0$  is some uniform constant which does not depend on $\xi$.
Note that
\begin{align}\label{eq:harmonic}
\lVert \varphi  \rVert_{L^2}:=\int_{0<|z_1|<\frac{1}{2}}|\varphi(z_1)|^2dz_1 d\bar{z}_1<C_4
\end{align}  
for some constant $C_4>0$. 
For any fixed $0\leq \xi_2,\ldots, \xi_n\leq \frac{1}{2}$,
consider the Dirichlet problem
\begin{align}
\begin{cases}
\phi=f_\xi \quad \mbox{on}\quad   \{z_1\mid |z_1|=\frac{1}{2}\}\\
\Delta \phi=\varphi\quad \mbox{on}\quad \{z_1\mid 0<|z_1|<\frac{1}{2}\}
\end{cases}
\end{align}
By \eqref{eq:harmonic} and the elliptic estimate,   one has
\begin{align}\label{eq:harmonic2}
\sup_{0<|z_1|<\frac{1}{2}}|\phi(z_1)|\leq C_5 (\lVert \varphi\rVert_{L^2}+\sup_{|z_1|=\frac{1}{2}} f_\xi). 
\end{align}
over $ \{z_1\mid 0<|z_1|<\frac{1}{2}\}$ for some uniform positive constant $C_5$ which does not depending on $\xi$. Hence $\Delta(f_\xi-\phi)\geq 0$ over $ \{z_1\mid 0<|z_1|<\frac{1}{2}\}$.  
 Since both $h$ and $\tilde{h}$ are adapted to log order, so is  $h_{\cF}$. Hence there is a constant $C_6>0$ so that
$$ 
  \log |s|^2_{h_\cF}\leq   C_6\log (-\sum_{i=1}^{\ell}\log |z_i|) 
$$ 
over $U^*(\frac{1}{2})$. By \cref{lem:extension2}, we conclude that $f_{\xi}-\phi$ extends to a subharmonic function on $ \{z_1\mid  |z_1|<\frac{1}{2}\}$. Note that   $f_\xi(z_1)-\phi(z_1)=0$ when $|z_1|=\frac{1}{2}$. Hence by maximum principle, 
$$
f_\xi(z_1)\leq \phi(z_1)
$$
for any $0<|z_1|<\frac{1}{2}$. Let 
$$
C_7:=\sup_{|z_1|=\frac{1}{2},0\leq \xi_2,\ldots,\xi_n\leq \frac{1}{2}} f_\xi(z_1)
$$
which is finite. By \eqref{eq:harmonic} and \eqref{eq:harmonic2}, we have
$$ 
\sup_{0<|z_1|<\frac{1}{2},0\leq z_2,\ldots,z_n\leq \frac{1}{2}} \log |s|^2_{h_\cF}(z_1, \ldots,z_n)\leq C_5(C_4+C_7).
$$ 
This implies that ${h}\geq C_8\cdot \tilde{h}$  over $U^*(\frac{1}{2})$ for some constant $C_8>0$. By \eqref{eq:bound2}, one has 
\begin{align*} 
|F_{\tilde{h}^*}(\cE^*)|_{h^*,\omega_e}^2\leq \frac{C_0}{(-\log |z_1|^2)^3|z_1|^2}.
\end{align*} 
Hence if  we use the   metric $ {h}\otimes \tilde{h}^{*}$ for $\cF$ and do the same proof, we can prove that
 ${h}\leq C_9\cdot \tilde{h}$  over $U^*(\frac{1}{2})$ for some constant $C_9>0$. Therefore, $\tilde{h}$ and $h$ are \emph{mutually bounded} on $U^*(\frac{1}{2})$.  
By
\begin{align} \label{eq:p1}
\tau (z_1\frac{\d}{\d z_1})&=e_1^*\otimes e_0  \\ \label{eq:p2}
\tau(\frac{\d}{\d z_j})&=e_j^*\otimes e_0 \quad \mbox{for} \  \quad j=2,\ldots,n,
\end{align}
we obtain the norm estimate for the metric
\begin{align}\label{eq:toroidal}
\tau^*h\sim \tau^*\tilde{h}= \frac{\sqrt{-1}dz_1\wedge d\bar{z}_1}{|z_1|^2(\log |z_1|^2)^2}+\sum_{k=2}^{n} \frac{\sqrt{-1}dz_k\wedge d\bar{z}_k }{-\log|z_1|^2} 
\end{align}
Though $\tau^*h$ is strictly less than the Poincar\'e metric near $D$, one can easily prove that it is still a   \emph{complete   metric}. 
Therefore, the hermitian metric $\tau^*h_H=  \tau^*h$ on $X-D$ is also  complete. Based on \cref{rem:complete}, we conclude that $X-D$ is uniformized by the complex unit ball of dimension $n$, namely, there is a torsion free lattice $\Gamma\subset PU(n,1)$ so that $X-D\simeq \faktor{\bB^n}{\Gamma}$. By \eqref{eq:p1} and \eqref{eq:p2}, the canonical K\"ahler-Einstein metric $\omega:=\tau^*h$ for $T_{X}(-\log D)|_{U}$ is  adapted to log order.  It   follows from \cref{thm:rigidity} that $X$ is the unique toroidal compactification for the non-compact ball quotient $\faktor{\bB^n}{\Gamma}$.  
 We accomplish the proof of \cref{item1}. 
\medskip

\noindent \textbf{Step 3.} 
Assume now  $D$ is not smooth.  By \eqref{eq:period2}, the period map $\widetilde{X-D}\to \bB^n$ is locally biholomorphic. Assume by contradiction that it is an isomorphism. Since $h$ is adapted to log order, the canonical K\"ahler-Einstein metric $\omega:=\tau^*h$ for $T_{X}(-\log D)|_{U}$ is also  adapted to log order. It   follows from \cref{thm:rigidity} that $D$ cannot be singular. The contradiction is obtained, and thus the period map is not a uniformizing mapping. We proved \cref{item2}.

Let us show that $K_{X}+D$ is big, nef and ample over $X-D$.   Note that the metric $\det \omega^{-1}$ for $(K_X+D)|_{U}$ is adapted to log order, and that
$$
\frac{\sqrt{-1}}{2\pi}R_{\det \omega^{-1}}((K_X+D)|_{U})=(n+1)\omega.
$$
By \cref{lem:current}, the   hermitian metric $\det \omega^{-1}$   extends to a singular hermitian metric $h_{K_X+D}$ for $K_X+D$ with zero Lelong numbers. Hence $K_X+D$ is nef. Since $\sn R_{h_{K_X+D}}(K_X+D)>0$ on $X-D$, $K_X+D$ is thus big and ample over $X-D$. 
We finish the proof of the theorem.
\end{proof}
\begin{rem}
Note that the asymptotic behavior of the metric \eqref{eq:toroidal} is exactly the same as that of the K\"ahler-Einstein metric for the ball quotient near the boundary of its toroidal compactification (see \cite[eq. (8) on  p. 338]{Mok12}). This is indeed the hint for our construction of  $\tilde{h}$.
\end{rem}
\begin{rem}
	We expect that \cref{item2} cannot  happen. This is   the case when $\dim X=2$. Indeed,  when the Miyaoka-Yau type equality in \eqref{eq:BM} holds, together with the conclusion that  $K_X+D$  is big, nef and ample over $X-D$ in \cref{thm:equality}, it follows from \cite{Kob85} that $X-D$ is uniformized by $\bB^2$, which is a contradiction to \cref{item2}. This is not surprising: consider the smooth toroidal compactification $X$ of a two dimensional ball quotient $\faktor{\bB^2}{\Gamma}$ with $D:=X-\faktor{\bB^2}{\Gamma}$,    \eqref{eq:equality} holds by \cref{main2}.    Let $x\in D$ and let $\pi:Y={\rm Bl}_xX\to X$. Then $Y$ is a projective surface compactifying $\faktor{\bB^2}{\Gamma}$ with the boundary $D_Y:=\pi^*D$   a simple normal crossing (not smooth) divisor. However, one has
	$$ 
		3c_2(\Omega_{Y}^1(\log D_Y))- c_1(\Omega_{Y}^1(\log D_Y))^2=1,
	$$  
	which violates  the condition  of uniformization in \cref{thm:equality}.
\end{rem}
\section{Higgs bundles associated to non-compact ball quotients}
In this section, we  will prove \cref{main2}. \cref{sec:positivity,sec:extension} are technical preliminaries. In \cref{sec:analysis} we prove that a log Higgs bundle $(E,\theta)$ on a compact K\"ahler log pair is slope polystable with respect to some polarization by big and nef cohomology $(1,1)$-class, if $(E,\theta)$ admits a  Hermitian-Yang-Mills metric with \enquote{mild singularity} near the boundary divisor. In \cref{sec:toroidal} we use the Bergman metric for  quotients of complex unit balls by   torsion free lattices   to construct such Hermitian-Yang-Mills metric. This proves \cref{main2}.
\subsection{Notions of  positivity for curvature tensors}\label{sec:positivity}
We recall some notions of  positivity for Higgs bundles in \cite[\S 1.3]{DH19}. 

Let $(E,\theta)$ be a Higgs bundle endowed with a smooth metric $h$. For any $x\in X$, let $e_1,\ldots,e_r$ be a frame of $E$ at $x$, and let $e^1,\ldots,e^r$  be its dual in  $E^*$. Let $z_1,\ldots,z_n$ be a local coordinate centered at $x$. We write
$$
F_h(E)=R_h(E)+[\theta,\overline{\theta}_h]=R_{j\bar{k}\alpha}^\beta dz_j\wedge d\bar{z}_k\otimes e^\alpha\otimes e_\beta
$$
Set  $R_{j\bar{k}\alpha\bar{\beta}}:=h_{\gamma\bar{\beta}}R_{j\bar{k}\alpha}^\gamma$, where $h_{\gamma\bar{\beta}}=h(e_\gamma,e_\beta)$.  $F_h(E)$ is called \emph{Nakano semi-positive} at $x$ if 
$$
\sum_{ j,k,\alpha,\beta}R_{j\bar{k}\alpha\bar{\beta}}u^{j\alpha} \overline{u^{k\beta}} \geq 0
$$ 
for any $u=\sum_{j,\alpha}u^{j\alpha}\frac{\partial}{\partial z_j}\otimes e_\alpha\in (T_{X}^{1,0}\otimes E)_x$.    $(E,\theta,h)$ is called Nakano  semipositive if $F_h(E)$ is  Nakano semi-positive  at every $x\in X$. When $\theta=0$, this reduces to the  same   positivity concepts in \cite[Chapter \rom{7}, \S 6]{Dembook} for vector bundles. 

We write 
$$F_h(E)\geq_{\nak} \lambda (\omega\otimes \vvmathbb{1}_E) \quad \mbox{ for } \lmd\in \mathbb{R}$$
if
$$
\sum_{ j,k,\alpha,\beta}(R_{j\bar{k}\alpha\bar{\beta}}-\lambda\omega_{j\bar{k}}h_{\alpha\bar{\beta}})(x)u^{j\alpha}\overline{u^{k\beta}}   \geq 0
$$
for    any $x\in X$ and any $u=\sum_{j,\alpha}u^{j\alpha}\frac{\partial}{\partial z_j}\otimes e_\alpha\in (T_{X}^{1,0}\otimes E)_x$. 
  
Let us recall the following lemma in \cite[Lemma 1.8]{DH19}.
\begin{lem}\label{lem:acceptable}
	Let $(E,\theta,h)$ be a Higgs bundle on a K\"ahler manifold $(X,\omega)$. If there is a positive constant $C$ so that $|F_h(x)|_{h,\omega}\leq C$ for any $x\in X$, then 
	$$
	C \omega\otimes \vvmathbb{1}_{E}\geq_{\nak} 	F_h\geq_{\nak} -C \omega\otimes \vvmathbb{1}_{E}.
	$$
\end{lem}
  The following easy fact  in \cite[Lemma 1.9]{DH19} will be useful in this paper.
 \begin{lem}\label{lem:tensor}
 	Let $(E_1,\theta_2,h_1)$ and $(E_2,\theta_2,h_2)$ are two metrized Higgs bundles over a K\"ahler manifold $(X,\omega)$ such that $|F_{h_1}(x)|_{h_1,\omega}\leq C_1$ and $|F_{h_2}(x)|_{h_2,\omega}\leq C_2$ for all $x\in X$. Then for the hermitian vector bundle $(E_1\otimes E_2,h_1h_2)$,  one has
 	$$
 	|F_{h_1\otimes h_2}(x)|_{h_1\otimes h_2,\omega}\leq \sqrt{2r_2C^2_1+2r_1C^2_2}
 	$$
 	for all $x\in X$. Here $r_i:=\rank E_i$.
 \end{lem}
\subsection{Some pluripotential theories}\label{sec:extension}
In this subsection we recall some results of deep pluripotential theories in \cite{BEGZ,Gue14}.  The results in this subsection will  be used in the proof of \cref{prop:analysis}. Let us first recall the definitions of big or nef cohomology $(1,1)$-classes in \cite[\S 6]{Dem12}.
\begin{dfn}\label{def:bignef}
	Let $(X,\omega)$ be a compact K\"ahler manifold. Let $\alpha\in H^{1,1}(X,\bR)$ be a cohomology $(1,1)$-class of $X$.  The class $\alpha$ is \emph{nef} (numerically eventual free) if for any $\ep>0$, there is a smooth closed $(1,1)$-form $\eta_\ep\in \alpha$ so that $\eta_\ep\geq -\ep \omega$. The class  $\alpha$ is \emph{big} if there is a closed positive $(1,1)$-current $T\in \alpha$ so that $T\geq \delta\omega$ for some $\delta>0$. Such  a current $T$ will be called a \emph{K\"ahler current}.
\end{dfn} 

Let $X$ be a complex manifold of dimension $n$ and let $U\subset X$ be a Zariski open set of $X$. Pick a smooth hermitian form $\omega$ on $X$. For any smooth differential form $\eta$ of degree $p$ on $U$ so that 
\begin{align*}
\int_{U} |\eta|_{\omega}\wedge \omega^{n} <+\infty,
\end{align*}
 one can \emph{trivially} extend $\eta$ to   \emph{a current $T_\eta$ on $X$ of degree $n-p$} by setting
\begin{align}\label{eq:trivial}
\langle T_\eta, u\rangle:=\int_{U}\eta \wedge u
\end{align}
where $u$ is the any \emph{test form} of degree $p$ which has compact support.  In general, $T_\eta$ might not be closed even if $\eta$ is closed.

Let $(X,\omega)$ be a compact K\"ahler manifold of dimension $n$. Let $\alpha_1,\ldots,\alpha_n$ be big   cohomology classes. Let $T_i\in \alpha_i$ be positive closed $(1,1)$-currents whose local potential is locally bounded outside a closed analytic subvariety  of $X$ (a particular case of \emph{small unbounded locus}  of \cite[Definition 1.2]{BEGZ}). In this celebrated work by Boucksom-Eyssidieux-Guedj-Zariahi \cite{BEGZ}, they defined  non-pluripolar product for these currents
$$
\langle T_1\wedge\cdots \wedge T_p\rangle
$$
which is a closed positive $(p,p)$-current, and does not charge on any closed proper analytic subsets. Therefore, if we assume further that $T_i$ is smooth over $X-A$ where $A$ is a closed analytic subvariety of $X$, then $\langle T_1\wedge\cdots \wedge T_p\rangle$ is nothing but the trivial extension of the $(p,p)$-form $ (T_1\wedge\cdots \wedge T_p)|_{X-A}$ to $X$.

Following \cite[Definition 1.21]{BEGZ}, for a big class $\alpha$, a positive $(1,1)$-current $T\in \alpha$ has \emph{full Monge-Amp\`ere mass} if
$$
\int_X \langle T_i^n\rangle =\vol (\alpha).
$$
The set of such positive currents in $\alpha$ with full Monge-Amp\`ere mass is denoted by $\cE(\alpha)$.
We will not recall the definition of the \emph{volume of big classes} by Boucksom in \cite{Bou02}. We just mention that when the class $\alpha$ is big and nef, one has
$$
\vol(\alpha)=\alpha^n.
$$

The following lemma will be used in \cref{sec:analysis}.
\begin{lem}\label{lem:Gue}
	Let $(X,\omega)$ be a compact K\"ahler manifold and let $D$ be a simple normal crossing divisor on $X$. Let $S$ be a closed positive $(1,1)$-current on $X$ so that $S|_{X-D}$ is a smooth $(1,1)$-form over $X-D$ which is strictly positive at one point  and has at most \emph{Poincar\'e growth} near $D$. Then the cohomology class $\alpha:=\{S\}$ is big and nef, and $S\in \cE(\alpha)$.
\end{lem}
\begin{proof}
	Let $T$ be the K\"ahler current on $X$ constructed in \cref{rem:Poincare}. Since $T|_{X-D}$ has at most Poincar\'e growth near $D$, there exists a constant $C_1>0$ so that
	$$
	C_1T-S\geq 0.
	$$
Pick any point $x\in D$. Then there exists some admissible coordinates $(U;z_1,\ldots,z_n)$ centered at $x$ so that the local potential $\varphi$ of $S$ satisfies that 
	$$\varphi\geq -C_1\log (-\prod_{i=1}^{\ell}\log |z_1|^2)-C_2$$
	for some constant $C_2>0$. Hence $S$ has zero Lelong numbers everywhere and thus $\alpha$ is nef. Since $S$ is strictly positive at one point on $X-D$, it is big by \cite{Bou02}. It follows from \cite[Proposition 2.3]{Gue14} that $S\in \cE(\alpha)$. The lemma is proved.
	\end{proof}
Let us recall an important theorem in \cite{BEGZ}.
 \begin{thm}[\!\!\protecting{\cite[Corollary 2.15]{BEGZ}}]\label{thm:DDL}
 	Let $(X,\omega)$ be a compact K\"ahler manifold of dimension $n$.  Let $\alpha_1,\ldots,\alpha_n$ be big  and nef classes on $X$.  For $T_i\in \cE(\alpha_i)$ which are all smooth outside a closed proper analytic subset $A$, one has
 	$$
 	\int_{X-A}T_1\wedge\cdots \wedge T_n=\int_X\langle T_1\wedge\cdots \wedge T_n\rangle=\alpha_1\cdots \alpha_n.
 	$$
 \end{thm}

\subsection{Hermitian-Yang-Mills metric and stability}\label{sec:analysis}
 Let $(X,\omega)$ be a compact K\"ahler manifold and let $D$ be a simple normal crossing divisor on $X$. For applications of birational geometry, one usually considers more general polarization by big and nef line bundles.  In this subsection, we will prove that  a log Higgs bundle $(E,\theta)$ on $(X,D)$ is $\mu_\alpha$-polystable if $(E,\theta)|_{X-D}$ admits a  Hermitian-Yang-Mills metric whose growth at infinity is \enquote{mild}, where $\alpha$ is certain big and nef cohomology class. When $\dim\, X=1$ or  $D=\varnothing$ and the polarization is K\"ahler, this has been proved by Simpson \cite{Sim88,Sim90}. As we have seen in \cref{thm:mochizuki}, when $X$ is projective and both the first and second Chern classes of $E$  vanish and the polarization is an ample line bundle, this result has been proved by Mochizuki.
 
 We start with the following technical result, which is strongly inspired by the deep result of Guenancia \cite[Proposition 3.8]{Gue17}.
\begin{proposition}\label{prop:analysis}
	Let $(X,\omega_0)$ be a compact K\"ahler manifold and let $D$ be a simple normal crossing divisor on $X$. Let $(E,\theta)$ be a log Higgs bundle on $(X,D)$. Let $\alpha$ be a big and nef cohomology $(1,1)$-class containing a positive closed $(1,1)$-current $\omega\in \alpha$ so that $\omega|_{X-D}$ is a smooth K\"ahler form  and has at most Poincar\'e growth near $D$. Assume that there is  a hermitian metric $h$ for $(E,\theta)|_{X-D}$ which is adapted to log order (in the sense of \cref{def:log order}) and is acceptable (in the sense of \cref{def:acceptable}).  Then for any saturated  Higgs subsheaf $G\subset E$, one has
	\begin{align}\label{eq:chern}
	c_1(G)\cdot  \alpha^{n-1}=\int_{X-D-Z}Tr(\sn R_{h_G}(G))\wedge \omega^{n-1} 
	\end{align}
	where $Z$ is the analytic subvariety of codimension at least two so that $G|_{X-Z}\subset E|_{X-Z}$ is a subbundle, and $h_G$ is the metric on $G$ induced by $h$.
\end{proposition}
\begin{proof}
	By \cref{rem:Poincare}, one can construct a \emph{K\"ahler current} \begin{align}\label{eq:metric2}
	T:=\omega_0-\hess \log (-\prod_{i=1}^{\ell}\log  |\ep\cdot \sigma_i|_{{h}_i}^2),
	\end{align}  over $X$, whose restriction on $X-D$ is a complete K\"ahler form $\omega_P$, which has the same Poincar\'e growth   near $D$.   Here $\sigma_i$ is the section $H^0(X,\cO_X(D_i))$ defining $D_i$, and $h_i$ is some smooth metric for the line bundle $\cO_X(D_i)$.   Since we assume that $h$ is acceptable, (after rescaling $T$ by multiplying a constant)  one thus has
	$$
	|F_h(E)|_{h,\omega_P}\leq 1.
	$$
	By \cref{lem:acceptable}, one has
	$$
	-\vvmathbb{1}\otimes \omega_P\leq_{Nak}F_h(E)\leq_{Nak}\vvmathbb{1}\otimes \omega_P
	$$
	over $X-D$. 
	
	We first consider the case that $G$ is an invertible saturated subsheaf of $E$ which is invariant under $\theta$.  Then the metric $h$ of $E$ induces a \emph{singular hermitian metric} $h_G$ for  $G$ defined on the whole $X$, which is smooth on on  $X^\circ:=X-D-Z$. The curvature current
	$\sn R_{h_G}(G)$ is a closed $(1,1)$-current on $X-D$, which is a smooth $(1,1)$-form on  $X^\circ$. Define by $\pi:E|_{X^\circ}\to G|_{X^\circ}$ the   orthogonal projection with respect to $h$ and $\pi^\perp:E|_{X^\circ}\to G^\perp|_{X^\circ}$ the projection to its   orthogonal complement. By the Chern-Weil formula (see for example  \cite[Lemma 2.3]{Sim88}), over $X^\circ $, we have
	\begin{align}\label{eq:second}
	R_{h_G}(G)=F_{h_G}(G)=F_h(E)|_{G}+\overline{\beta}_h\wedge \beta-\varphi\wedge \overline{\varphi}_h
	\end{align}
	where $F_h(E)|_{G}$ is the orthogonal projection of $F_h(E)$ on $\hom(G,G)|_{X^\circ}=\cO_{X^\circ}$, and $\beta\in \sA^{1,0}(X^\circ,\hom(G,G^\perp))$ is the second fundamental form, and $\varphi\in\sA^{1,0}(X^\circ,\hom(G^\perp,G))$ is equal to $ \theta|_{G^\perp}$. Hence $\sn R_{h_G}(G)\leq \sn F_h(E)|_{G}$.

	
	For any local frame $e$ of $G|_{X^\circ}$,  note that
	$$
	|e|_{h}^2\cdot \sn F_{h}(E)|_{G}=\langle \sn F_{h}(E)(e), e\rangle_{h}\leq \langle \vvmathbb{1}\otimes \omega_P e, e\rangle_{h}=|e|_{h}^2\cdot \omega_P
	$$
	Hence $\sn F_{h}(E)|_{G}-\omega_P$ is a semi-negative $(1,1)$-form on $X^\circ$, and thus   over $X^\circ$ one has
\begin{align*} 
	- \sn R_{h_G}(G)+T\geq \omega_P -\sn F_{h}(E)|_{G} \geq 0 
\end{align*}
	Since we assume that $(E,h)$ is adapted to log order,   $(G^{-1}|_{X-Z},h^{-1}_{G}|_{X-Z})$ is thus  adapted to log order for the log pair $(X-Z,D-Z)$.     By \cref{lem:current} and \eqref{eq:metric2}, $- \sn R_{h_G}(G)+T$ extends to a closed positive $(1,1)$-current on $X-Z$.  
	Since $Z$ is of codimension at least two, 
	a  standard fact in pluripotential theory (see \cite[Theorem 3.3.42]{NO90}) shows that   $- \sn R_{h_G}(G)+T$  extends to a  positive closed $(1,1)$-current  on the whole $X$.

Denote by $s\in H^0(X,E\otimes G^{-1})$ the section defining the inclusion $G\to E$. 	We fix a smooth hermitian metric $h_0$ for $G$ and we define a function $H:=|s|^2_{h\cdot h_0^{-1}}=h_G\cdot h_0^{-1}$ on $X-D$. Then
\begin{align}\label{eq:difference}
	\hess \log H=\sn R_{h_0}(G)-\sn R_{h_G}(G). 
	\end{align}
Hence there is a constant $C_0>0$ so that
	\begin{align}\label{eq:difference2}
	\hess \log H+C_0T\geq T. 
	\end{align}
	By \cref{lem:Gue}, $\omega\in \cE(\alpha)$. Since $\sn R_{h_0}(G)$ is a smooth $(1,1)$-form on $X$, it follows from \cref{thm:DDL} that
	$$
	\int_{X^\circ}\sn R_{h_0}(G)\wedge \omega^{n-1}=c_1(G)\cdot \alpha^{n-1}.
	$$
	To prove \eqref{eq:chern},	by \eqref{eq:difference} and the above equality it suffices to prove that 
	\begin{align}\label{eq:suffice}
	\int_{X^\circ}\hess \log H\wedge \omega^{n-1}=0.
	\end{align} 
	We will pursue the ideas in \cite[Proposition 3.8]{Gue17} to prove this equality.
	
Let us take a log resolution $\mu:\tilde{X}\to X$ of the ideal sheaf $\sI$ defined by $s\in H^0(X,E\otimes G^{-1})$, with $\cO_{\tilde{X}}(-A)=\mu^*\sI$ and $\tilde{D}:=\mu^{-1}(D)$ a   simple normal crossing divisor.  Let us denote by $(\tilde{E},\tilde{\theta})$ the induced log Higgs bundle on $(\tilde{X},\tilde{D})$ by pulling back $(E,\theta)$ via $\mu$. Then the metric $\tilde{h}:=\mu^*h$ for $(\tilde{E},\tilde{\theta})|_{\tilde{X}-\tilde{D}}$ is also adapted to log order and acceptable by \cref{lem:birational}.  
	
	Note that ${\rm Supp}(\cO_X/\sI)=Z$. Write $\tilde{G}:=\mu^*G$. There is a nowhere vanishing section
	$$
	\tilde{s}\in H^0(\tilde{X},\tilde{E}\otimes \tilde{G}^{-1}\otimes \cO_{\tilde{X}}(-A))
	$$
	so that $\mu^*s=\tilde{s}\cdot \sigma_{A}$, where $\sigma_A$ is the canonical section in $H^0(\tilde{X},\cO_{\tilde{X}}(A))$ which defines the effective exceptional divisor $A$.
	
	Fix a K\"ahler form $\tilde{\omega}$ on $\tilde{X}$, as \cref{rem:Poincare} we construct another  {K\"ahler current} \begin{align}\label{eq:metric3}
	\tilde{T}:=\tilde{\omega}-\hess \log (-\prod_{i=1}^{m}\log  |\ep\cdot \tilde{\sigma}_i|_{\tilde{h}_i}^2),
	\end{align}  over $\tilde{X}$, whose restriction on $\tilde{X}-\tilde{D}$ is a complete K\"ahler form,  which has the same Poincar\'e growth   near $\tilde{D}$. Here $\tilde{\sigma}_i$ is the section $H^0(X,\cO_X(\tilde{D}_i))$ defining $\tilde{D}_i$, and $\tilde{h}_i$ is some smooth metric for the line bundle $\cO_{\tilde{X}}(\tilde{D}_i)$.

	Let us fix a smooth hermitian metric $h_A$ for $\cO_{\tilde{X}}(A)$. 	Write $\tilde{H}:=|\tilde{s}|_{\tilde{h}\cdot \mu^*h_0^{-1}\cdot h_A^{-1}}^2$. Since $\tilde{h}$ is adapted to log order and $\tilde{s}$ is nowhere vanishing, there is a constant $C_1, C_2>0$ so that
	\begin{align}\label{eq:potential}
	\log \tilde{H}\geq C_1\varphi_P-C_2,
	\end{align}
	where we denote by $$\varphi_P:=-\log (-\prod_{i=1}^{\ell}\log  |\ep\cdot \tilde{\sigma}_i|_{\tilde{h}_i}^2).$$  
	Since $\tilde{h}:=\mu^*h$ for $(\tilde{E},\tilde{\theta})|_{\tilde{X}-\tilde{D}}$ is  acceptable, by  same arguments as those for \eqref{eq:difference2}, one can show that
	$$
	\hess \log \tilde{H}+C_3\tilde{T}\geq \tilde{T}
	$$
	over $\tilde{X}-\tilde{D}$ for some constant $C_3 >0$. Note that  the local potential of $	\hess \log \tilde{H}+C_3\tilde{T}$ is bounded from below by $(C_1+C_3)\varphi_P$ according to \eqref{eq:potential}. By \cite[Proposition 2.3]{Gue14}, one has 
	$$
		\hess \log \tilde{H}+C_3\tilde{T}\in \cE(\{C_3\tilde{T}\}).
	$$	
One can check that $\mu^*\omega\leq C_4\tilde{T}$ for some constant $C_4>0$. By \cref{lem:Gue} again, $\mu^*\omega\in \cE(\mu^*\alpha)$.
   Hence by \cref{thm:DDL} one has 
	$$
	\int_{\mu^{-1}(X^\circ)}(\hess \log \tilde{H}+C_3\tilde{T})\wedge \mu^*\omega^{n-1}=\{C_3\tilde{T}\}\cdot \mu^*\alpha^{n-1}.
	$$
	Recall that $\tilde{T}\in \cE(\{\tilde{T}\})$ by \cref{lem:Gue}.
Hence
	$$
	\int_{\mu^{-1}(X^\circ)} C_3\tilde{T}\wedge \mu^*\omega^{n-1}=\{C_3\tilde{T}\}\cdot \mu^*\alpha^{n-1}.
	$$
One thus has
\begin{align}\label{eq:1}
	\int_{\mu^{-1}(X^\circ)} \hess \log \tilde{H} \wedge \mu^*\omega^{n-1}=0.
\end{align}
	Note that over $\tilde{X}-\tilde{D}$, one has
	$$\hess \log \tilde{H}+[A]-\sn R_{h_A}(A)=\mu^*\hess \log H$$
	where $[A]$ is the current of integration of $A$. Hence over $\mu^{-1}(X^\circ)\simeq X^\circ$, one has
\begin{align}\label{eq:2}\hess \log \tilde{H}-\sn R_{h_A}(A)=\mu^*\hess \log H.
\end{align}
By \cref{thm:DDL} again,
\begin{align}\label{eq:3}
	\int_{\mu^{-1}(X^\circ)} \sn R_{h_A}(A) \wedge \mu^*\omega^{n-1}=c_1(A)\cdot \mu^*\alpha^{n-1}=0,
\end{align}
where the last equality follows from the fact that $A$ is $\mu$-exceptional. \eqref{eq:1}, \eqref{eq:2} together with \eqref{eq:3} shows the desired equality \eqref{eq:suffice}.   We finish the proof of \eqref{eq:chern} when $\rank\, G=1$.

	Assume that $\rank\, G=r$. We replace $(E,\theta,h)$ by  the wedge product $(\tilde{E},\tilde{\theta},\tilde{h}):=\Lambda^r(E,\theta,h)$. By \cref{lem:tensor}, the induced metric $\tilde{h}$ is also acceptable and one can easily check that it is also adapted to log order. Note that $\det G$ is also invariant under $\tilde{\theta}$, and that
 $
	\det G\to \Lambda^rE$. 
	We then reduce the general cases to rank $1$ cases. The proposition is thus proved.
\end{proof}

Let us state and prove the main result in this section.
\begin{thm}\label{thm:criteria}
	Let $X$ be a compact K\"ahler manifold and let $D$ be a simple normal crossing divisor on $X$. Let $\alpha$ be a big and nef cohomology $(1,1)$-class containing a positive closed $(1,1)$-current $\omega\in \alpha$ so that $\omega|_{X-D}$ is a smooth K\"ahler form  and has at most Poincar\'e growth near $D$.  Let $(E,\theta)$ be a log Higgs bundle on $(X,D)$.   Assume that there is  a hermitian metric $h$ on $(E,\theta)|_{X-D}$ such that
	\begin{itemize}[leftmargin=0.4cm]
		\item it is  adapted to log order (in the sense of \cref{def:log order});
		\item it is acceptable (in the sense of \cref{def:acceptable});
		\item it is \emph{Hermitian-Yang-Mills}:
		$$
		\Lambda_{\omega}F_h(E)^{\perp}=0.
		$$
	\end{itemize} 
	Then $(E,\theta)$ is $\mu_\alpha$-polystable.
\end{thm}
\begin{proof}
	We shall use the same notations as those in \cref{prop:analysis}. 
	Let $G$ be any saturated  Higgs-subsheaf $G\subset E$, 	and denote by $Z$ the analytic subvariety of codimension at least two so that $G|_{X-Z}\subset E|_{X-Z}$ is a subbundle. By the Chern-Weil formula again, over $X^\circ:=X-Z-D$  we have
	\begin{align*} 
\Lambda_{\omega}F_{h_G}(G) 
 =\frac{\Lambda_{\omega} Tr (F_h(E))}{\rank\, E}\otimes \vvmathbb{1}_{G} +\Lambda_{\omega}(\overline{\beta}_h\wedge \beta-\varphi\wedge \overline{\varphi}_h).
	\end{align*}
where $\beta\in \sA^{1,0}(X^\circ,\hom(G,G^\perp))$ is the second fundamental form of $G$ in $E$ with respect to the metric $h$, and $\varphi\in\sA^{1,0}(X^\circ,\hom(G^\perp,G))$ is equal to $ \theta|_{G^\perp}$. 
	
	Hence
	\begin{align*}
 \int_{X^\circ} Tr (\sn F_{h_G}(G))\wedge  \omega^{n-1} =\int_{X^\circ}\frac{\rank\, G }{\rank\, E} Tr (\sn F_h(E))  \wedge  \omega^{n-1} -(|\beta|_h^2+|\varphi|_h^2)  \frac{\omega^{n}}{n}.
	\end{align*}
	By \cref{prop:analysis} together with the above inequality, one concludes the slope inequality
	$$
	\mu_\alpha(G)\leq\mu_\alpha(E)
	$$
	and the equality holds if and only if $\beta\equiv 0$ and $\varphi\equiv 0$. We shall prove that if the above slope equality holds, $G$ is a \emph{sub-Higgs bundle} of $E$, and we have the decomposition
	$$
	(E,\theta)=(G,\theta|_{G})\oplus (F,\theta_F)
	$$
	where $(F,\theta_F)$ is another sub-Higgs bundle of $E$.

 Set $\rank\, E=r$ and $\rank\, G=m$. We first prove that $G$ is a subbundle of $E$. It is equivalent to show that $\det G\to \Lambda^{r}E$ is a subbundle, and we thus reduce the problem to the case that $\rank\, G=1$. Assume that 	 $
 \mu_\alpha(G)=\mu_\alpha(E)
 $ 
 and thus  $\beta\equiv 0$ and $\varphi\equiv 0$. By \eqref{eq:second}, over $X^\circ$ one has
\begin{align} \label{eq:converse}
\sn R_{h_G}(G)=\sn F_h(E)|_{G} \geq -T|_{X^\circ},
\end{align} 
where $T$ is the K\"ahler current defined in \eqref{eq:metric2}. By \cref{lem:current}, $\sn R_{h_G}(G)+T$ extends to a closed positive $(1,1)$-current on   $X-Z$, and thus to  the whole $X$. 

 Assume now $x_0\in X$ is a point where $(E/G)_{x_0}$ is not locally free. Take a local holomorphic frame $e$ of $G$ on some open neighborhood $(U;z_1,\ldots,z_n)$ of $x$, and a holomorphic frame $e_1,\ldots,e_r$ of $E$. Then $e=\sum_{i=1}^{r}f_i(x)e_i$, where $f_i\in \cO(U_i)$ so that $f_1(x_0)=\cdots=f_r(x_0)=0$. By the asssumption that  $h$ is adapted to log order, one concludes that
\begin{align} \label{eq:contra}
\log |e|_h^2\leq C_1\log (|z_1|^2+\cdots+|z_n|^2)+C_2\log(-\log (\prod_{i=1}^{\ell}|z|_i^2)) 
\end{align} 
for some positive constants $C_1$ and $C_2$. On the other hand, by \eqref{eq:converse} on $U$ we have
$$
 \hess \log |e|_h^2=-\sn R_{h_G}(G)\leq T.  
$$
By the construction of $T$, we conclude that
$$
\log |e|_h^2\geq  -C_3\log(-\log (\prod_{i=1}^{\ell}|z|_i^2))-C_4,
$$
for some positive constants $C_3$ and $C_4$. This contradicts with \eqref{eq:contra}. Hence we conclude that when the slope equality holds, $G$ is a   subbundle of $E$.

We now find the desired decomposition of $(E,\theta)$. By the above argument, when the slope equality holds, $(G,\theta|_G)$ is a Higgs subbundle of $(E,\theta)$ (not assumed to be rank 1 now), and  $\beta\equiv 0$ and $\varphi\equiv 0$.  This means that the  orthogonal projection $\pi: E |_{X-D}\to  G |_{X-D}$ is holomorphic, that $G^\perp$ is a holomorphic subbundle of $E|_{X-D}$, and that \begin{align}\label{eq:splits}
(E,\theta)|_{X-D}=(G,\theta|_{G})|_{X-D}\oplus (G^\perp,\theta|_{G^\perp}).
\end{align}
 We shall prove that $\pi$ extends to a morphism $\tilde{\pi}:E\to G$ so that $\pi\circ \iota=\vvmathbb{1}$.  For any point $x_0\in D$, we pick an admissible coordinate $(U;z_1,\ldots,z_n)$ centered at $x_0$ and a holomorphic fame $(e_1,\ldots,e_r)$ for $E|_U$ adapted to log order so that $(e_1,\ldots,e_m)$ is  a holomorphic fame  for $G|_U$. Write $\pi(e_j|_{X-D})=\sum_{i=1}^{r}f_{i}(x)e_i$, where $f_i(x)\in \cO(U-D)$. For $j=1,\ldots,m$, one has $\pi(e_j|_{X-D})=e_j$ and it extends naturally.  For $j>m$ and some $1<r<1$,  over $U^*(r)$ one has
$$
C(-\log (\prod_{i=1}^{\ell}|z|_i^2))^M\geq   |e_j|_h^2\geq |\pi(e_j)|_h^2\geq   C^{-1}(-\log (\prod_{i=1}^{\ell}|z|_i^2))^{-M}\sum_{i=1}^{r}|f_i|^2
$$
for some $C,M>0$, where   the second inequality is due to the fact  that $\pi$  is the orthogonal projection with respect to $h$, and  the last inequality follows from the fact that $h$ is adapted to log order.   Hence each $|f_i|$ must be locally bounded from above on $U$, and it thus extends to a holomorphic function on $U$. We conclude that $\pi$ extends to a morphism $\tilde{\pi}:E\to G$, whose rank is constant and $\tilde{\pi}\circ\iota=\vvmathbb{1}$, where $\iota:G\to E$ denotes the inclusion. Let us define by $F:=\ker \tilde{\pi}$, which is a subbundle of $E$ so that $E=G\oplus F$.   Note that $F|_{X-D}=G^\perp$.  By \eqref{eq:splits} together with the continuity propery we conclude that $F$ is a sub-Higgs bundle of $(E,\theta)$, and that $(E,\theta)=(G,\theta|_G)\oplus (F,\theta|_F)$. Since $h|_{G}$ (resp. $h|_{F}$) is a Hermitian-Yang-Mills metric for $(G,\theta|_G)$ (resp.  $(F,\theta|_F)$) satisfying the three conditions in the theorem, we can argue in the same way as above   to decompose   $(G,\theta|_G)$ and  $(F,\theta|_F)$ further to show that $(E,\theta)$ is a direct sum of  $\mu_\alpha$-stable log Higgs bundles with the same slope. Hence  $(E,\theta)$ is  $\mu_\alpha$-polystable. We prove  the theorem.
\end{proof}

 \subsection{Application to toroidal compactification of ball quotient}\label{sec:toroidal}
Let $\Gamma\in PU(n,1)$ be a torsion free lattice, and let $\faktor{\mathbb{B}^n}{\Gamma}$  be the associated  ball quotient. By the work of Baily-Borel, Siu-Yau and Mok \cite{Mok12},   $\faktor{\mathbb{B}^n}{\Gamma}$ has a unique structure of
a quasi-projective complex algebraic variety (see for example \cite[Theorem 3.1.12]{BU20}).  When the parabolic subgroups of $\Gamma$ are unipotent, by the work of Ash et al.  \cite{AMR10} and Mok \cite[Theorem 1]{Mok12}, $\faktor{\mathbb{B}^n}{\Gamma}$  admits a \emph{unique} smooth toroidal compactification, which we denote by $X$. Let us denote by $D:=X-\faktor{\mathbb{B}^n}{\Gamma}$  the boundary divisor, which is a disjoint union of abelian varieties. Let $g_B$ be the Bergman metric for $\mathbb{B}^n$, which is complete, invariant under $PU(n,1)$ and has constant holomorphic sectional curvature $-1$. Hence it descends to a metric $\omega$ on $X-D$.  If we consider $\omega$ as a metric for $T_{X}(-\log D)|_{X-D}$, by \cite[Proposition 2.1]{To93} it is \emph{good} in the sense of Mumford \cite[Section 1]{Mum77}.  Therefore, for any $k\geq 1$, it follows from \cite[Theorem 1.4]{Mum77} that the trivial extension of the  Chern form $c_k(T_{X-D},\omega)$ onto $X$ defines a $(k,k)$-current $[c_k(T_{X-D},\omega)]$ on $X$, which represents the cohomology class $c_k(T_X(-\log D))\in H^{k,k}(X)$.     Let us first prove \eqref{eq:equality}, which is indeed an easy computation.

For any $x_0\in X-D$, we take a normal coordinate system $(z_1,\ldots,z_n)$ centered at $x_0$ so that
$$
\omega=\sn \sum_{1\leq \ell,m\leq n}\delta_{\ell m}dz_\ell\wedge d\bar{z}_m-\sum_{j,k,\ell,m}c_{jk\ell m} z_j \bar{z}_k+O(|z|^3)
$$
where $c_{jk\ell m}$ is the coefficients of the Chern curvature tensor
$$
R_{\omega}(T_X)=\sum_{j,k,\ell,m}c_{jk\ell m}dz_j\wedge d\bar{z}_k\otimes (\frac{\d}{\d z_\ell})^*\otimes \frac{\d}{\d z_m}.
$$
By \cite[p. 177]{Mok89}, one has
\begin{align}\label{eq:ball}
c_{jk\ell m}(x_0)=-(\delta_{jk}\delta_{\ell m}+\delta_{jm}\delta_{k\ell}).
\end{align}
One can check that
$$
nc_1(T_{X-D},\omega)^2-2(n+1)c_2(T_{X-D},\omega)\equiv 0.
$$
We thus conclude that  the Chern classes   $c_k(\Omega^1_X(\log D))$ satisfies
$$
nc_1(\Omega_X^1(\log D))^2-2(n+1)c_2(\Omega_X^1(\log D))=0.
$$
Hence \eqref{eq:equality} in \Cref{main2} holds.

For the log Hodge bundle $(E,\theta)=(E^{1,0}\oplus E^{0,1},\theta)$, given by
$$
E^{1,0}:=\Omega^1_X(\log D), \quad E^{0,1}:=\cO_X
$$ 
with the Higgs field $\theta$ defined in \eqref{eq:Higgs},  we shall prove that it is $\mu_\alpha$-polystable for the big and nef polarization $\alpha$  in \cref{thm:criteria}. We equipped $(E^{1,0}\oplus E^{0,1})|_{X-D}$ with the metric
 \begin{align}\label{eq:metric harmonic}
h:=\omega^{-1}\oplus h_c
\end{align} 
where $h_c$ is the canonical metric on $\cO_{X-D}$ so that $|1|_{h_c}=1$.  Recall that the curvature $F_h(E)$ of the connection $D_h:=d_h+\theta+\overline{\theta}_h$ is
$$
F_h(E)=R_h(E)+[\theta,\overline{\theta}_h],
$$
where $R_h(E)$ is the Chern curvature of $(E,h)$. An easy exercise shows that
$$
\sn F_h(E)=\omega\otimes \vvmathbb{1}.
$$ 
In particular, $h$  is a  Hermitian-Yang-Mills metric for $(E,\theta)|_{X-D}$. We shall show that it satisfies the three conditions in \cref{thm:criteria}. Indeed, we only have to check the first two conditions since $\sn F_h(E)^\perp\equiv 0$.

We first note that $\omega$ has at most Poincar\'e growth near $D$ in the sense of \cref{def:Poincare}. Indeed, this follows easily from the Ahlfors-Schwarz lemma (see for example \cite[Lemma 2.1]{Nad89}) since the holomorphic sectional curvature of $\omega$ is $-1$. Hence for any admissible coordinate system $(U;z_1,\ldots,z_n)$ as in \cref{def:admissible},  one has $|F_h(E)|_{h,\omega_P}\leq C$, where $\omega_P$ is the Poincar\'e metric on $U^*$. 


By the following result, we see that $h$ is adapted to log order.
\begin{lem}[\protecting{\cite[eq. (8) on  p. 338]{Mok12}}]\label{lem:To}
	Let $(X,D)$ be as above. Then for any $x\in D$, there is an admissible coordinate $(U;z_1,\ldots,z_n)$ at $x$ so that  the frame $z_1\frac{\d}{\d z_1},\frac{\d}{\d z_2},\ldots,\frac{\d}{\d z_{n-1}}, \frac{\d}{\d z_n}$ is adapted to  log order (in the sense of \cref{sec:adapt}) with respect to the above metric $\omega$.
\end{lem}
Therefore,  the metric $h$ for $(E,\theta)|_{X-D}$ satisfies the three conditions in \cref{thm:criteria}.  In conclusion, $(E,\theta)$ is  $\mu_{\alpha}$-polystable for the big and nef class $\alpha$ in \cref{thm:criteria}

To finish the proof of   \cref{main2}, we have to show that $c_1(K_{X}+D)$ can be made as a polarization in \cref{thm:criteria}, which follows from the following result.
\begin{lem}[\protecting{\cite[Proposition 1]{Mok12}}]\label{lem:final}
The K\"ahler form $\frac{(n+1)}{2\pi}\omega$ on $X-D$ defined above extends to a closed positive $(1,1)$-current  $\varpi\in   c_1(K_X+D)$ with zero Lelong numbers. In particular, $K_X+D$ is big and nef.
\end{lem} 
\subsection{Proof of \cref{corx}}
We shall show how to apply \cref{main,main2} to derive \cref{corx}. 
\begin{proof}[Proof of \cref{corx}]
	 We first assume that  parabolic subgroups of $\Gamma$ are unipotent.	 
	 By \cite[Theorem 1]{Mok12},  there is   a toroidal compactification $\overline{X}$ for the ball quotient $X:=\faktor{\bB^n}{\Gamma}$, so that $D:=\overline{X}-X$ is a smooth divisor.   Moreover, $\overline{X}$ is projective.  Fix any ample polarization $L$ on $X$.   By \cref{main2},  the log Higgs bundle 
	 $(E,\theta):=(\Omega_{\overline{X}}^1(\log D)\oplus \cO_{\overline{X}},\theta)$ on $(\overline{X},D)$ defined in \eqref{eq:Higgs} is $\mu_L$-polystable with 	
	 \begin{align} \label{eq:chern2}
	 	2c_2(\Omega_{\overline{X}}^1(\log D))-\frac{n}{n+1}c_1(\Omega_{\overline{X}}^1(\log D))^2=0. 
	 \end{align} 
 Let us denote by $\overline{X}^\sigma$ and $D^\sigma$ the conjugate varieties of $\overline{X}$  and $D$ under $\sigma$. Hence $\overline{X}^\sigma$ is a smooth projective variety and $D^\sigma$  is a smooth divisor on $\overline{X}^\sigma$. For any coherent sheaf $\cE$ on $\overline{X}^\sigma$,  we denote by $\cE^\sigma$ its conjugate under $\sigma$, which is also a coherent sheaf on $\overline{X}^\sigma$.   Note that the conjugate action induces a canonical  isomorphism between cohomology groups \begin{align*} 
 	\Phi_k:H^{2k}(\overline{X}, \bC)\stackrel{\simeq}{\to}H^{2k}(\overline{X}^\sigma, \bC),
 	\end{align*}
and that Chern classes of vector bundles over $\overline{X}$ are preserved under  $\sigma$ in the sense that $
 	\Phi_k(c_k(F))=c_k(F^\sigma)$ for any holomorphic vector bundle $F$ over $\overline{X}$.   Since $\big(\Omega_{\overline{X}}^1(\log D)\big)^\sigma=\Omega_{\overline{X}^\sigma}^1(\log D^\sigma)$,
\eqref{eq:chern2} also holds for the log cotangent bundle $\Omega_{\overline{X}}^1(\log D)$. Moreover, the conjugate of $(E,\theta)$ under $\sigma$ is the log Higgs bundle $(E^\sigma,\theta^\sigma):=(\Omega_{\overline{X}^\sigma}^1(\log D^\sigma)\oplus \cO_{X^\sigma},\theta^\sigma)$ on $(\overline{X}^\sigma,D^\sigma)$ defined in \eqref{eq:Higgs}. Hence for any Higgs subsheaf $\cF$ of $(E,\theta)$, $\cF^\sigma$ is also a Higgs subsheaf of $(E^\sigma,\theta^\sigma)$ with the slope inequality $\mu_{L}(\cF)=\mu_{L^\sigma}(\cF^\sigma)$.  Hence $(E^\sigma,\theta^\sigma)$  is $\mu_{L^\sigma}$-polystable.  By \cref{main}, $\overline{X}^\sigma-D^\sigma$ is also a ball quotient, with $\overline{X}^\sigma$ its toroidal compactification.

For a general torsion free lattice $\Gamma\subset PU(n,1)$, there is a finite index subgroup $\Gamma'\subset \Gamma$ so that parabolic subgroups of $\Gamma'$ are unipotent (see for example \cite[\S 3.3]{BU20}). Denote by $X:=\faktor{\bB^n}{\Gamma}$ and $Y:=\faktor{\bB^n}{\Gamma'}$.      Since the base change of  an \'etale morphism   is étale,  we conclude that $Y^\sigma\to X^\sigma$ is also  a finite \'etale  surjective morphism. By the above result, $Y^\sigma$ is the ball quotient, and thus so is  $X^\sigma$.   \cref{corx} is proved.
 	\end{proof}
 \begin{rem}
 In the above proof we  show that a toroidal compactification of a ball quotient is also a toroidal compactification of another ball quotient.  As pointed out by the referee, this fact follows from \cref{corx} directly. His/her elegant argument is as follows: since an Abelian variety is simply an algebraic variety with a group law deﬁned by regular functions, a Galois conjugate of the Abelian variety as a component of the compactifying divisor is by itself an Abelian variety, which sits in the compactifying divisor of the ball quotient obtained as the conjugate of the original one according to \cref{corx}. 
 \end{rem}
\appendix
 
 \section{Metric rigidity for toroidal compactification of non-compact ball quotients}
 \smallskip
 \begin{center}by \textsc{Beno\^it Cadorel and Ya Deng}\end{center}
 \bigskip
 
 The main motivation of this appendix is to provide one building block for \cref{item2}. Our main result, \cref{thm:rigidity},  says   that there is no other   smooth compactification for non-compact ball quotient than the toroidal one, so that  the Bergman metric  grows \enquote{mildly} near the boundary. 
 Besides its own interests, this result is applied to show that the smoothness of $D$ in \cref{main} is necessary if one would like to characterize non-compact ball quotients.

 \subsection{Toroidal compactifications of quotients by non-neat lattices} \label{sect:torcomp}
 
 In this section, we recall a well known way of constructing the toroidal compactifications of ball quotients in the case where the lattice has torsion at infinity.  The reader will find more details about the natural orbifold structure on these compactifications in \cite{eys18}. For our purposes, the basic result given in Proposition \ref{proptorcom} will be sufficient.
 \medskip
 
 Recall that we say that a lattice $\Gamma \subset PU(n,1)$ is {\em neat} (cf. \cite{Bor69}) if for any $g \in \Gamma$, the subgroup of $\mathbb C^\ast$ generated by the eigenvalues of $g$ is torsion free. This implies that $\Gamma$ is torsion free and that all parabolic elements of $\Gamma$ are unipotent, so that the toroidal compactifications of $\faktor{\mathbb B^n}{\Gamma}$ provided by \cite{AMR10, Mok12} are {\em smooth} (there is no "torsion at infinity"). Note that by \cite[Proposition 17.4]{Bor69} in the arithmetic case, and \cite{Bor63}, or \cite[Theorem 6.11]{Rag72} in general, any lattice in $PU(n,1)$ admits a finite index neat normal sublattice.

 \begin{proposition} \label{proptorcom}
 	Let $\Gamma \subset PU(n, 1)$ be a torsion free lattice, and let $\Gamma' \subset \Gamma$ be a finite index normal \emph{neat} sublattice. Let $U = \faktor{\mathbb B^n}{\Gamma}$, $U' = \faktor{\mathbb B^n}{\Gamma'}$, and denote by $X'$ the smooth toroidal compactification of $U' = \faktor{\mathbb B^n}{\Gamma'}$ as constructed in \cite{AMR10, Mok12}. 
 	
 	Then the natural action of the finite group $G = \faktor{\Gamma}{\Gamma'}$ on $U'$ extends to $X'$, and the quotient $X = \faktor{X'}{G}$ is a normal projective space, with boundary $X - U$ made of quotient of abelian varieties by finite groups. Moreover, when $\Gamma$ is arithmetic, $X$ coincides with the toroidal compactification of $U$ constructed in \cite{AMR10}.
 \end{proposition}
 
 Before explaining how to prove Proposition \ref{proptorcom}, let us recall the construction of $X'$ as it is defined in \cite{Mok12} (see also \cite{cad16} for a similar discussion).
 
 Each component $D$ of $X' - U'$ is associated to a certain $\Gamma'$-orbit of points of $\partial \mathbb B^n$, whose points are called the $\Gamma'$-{\em rational boundary components} of $\partial B^n$ (cf. \cite[Chapter 3]{AMR10} or \cite[\S 1.3]{Mok12}). Let $b \in \partial \mathbb B^n$ be such a point, and let $N_b \subset PU(n,1)$ be its stabilizer. This is a maximal parabolic real subgroup of $PU(n, 1)$; let us denote by $W_b$ its unipotent radical. This group is an extension $1 \to U_b \to W_b \overset{\pi}{\to} A_b \to 1$, where $A_b \cong \mathbb C^{n-1}$, and $U_b \cong \mathbb R$ is the center of $W_b$. Let $L_b = \faktor{N_b}{W_b}$. This reductive group can be embedded as a Levi subgroup in $N_b$, so that $N_b = W_b \cdot L_b$. Moreover, we have a further decomposition $L_b = U(n-1) \times \mathbb R$. 
 (all this description can be obtained e.g. by specializing the discussion of \cite[Section 1.3]{BB66} or \cite[Section 4.2]{AMR10} to the case of the ball). 
 \medskip
 
 This Lie theoretic description of $N_b$ can be understood more easily by expressing the action of the previous groups on the horoballs tangent to $b$. Let $(S_b^{(N)})_{N \geq 0}$ be the family of these horoballs. Each $S_b^{(N)} \subset \mathbb B^n$ can be described as an open subset in a Siegel domain of the third kind, as follows:
 \begin{equation} \label{eq:exprsn}
 	S_b^{(N)} \simeq \{ (z', z_n) \in \mathbb C^{n-1} \times \mathbb C \; | \; {\rm Im}\, z_n > || z ' ||^2 + N \}.
 \end{equation}
 We have $S_b^{(0)} \cong \mathbb B^n$, and when $b = (0, ..., 0, 1)$, the change of coordinates between the two descriptions of the ball is given by the {\em Cayley transform}
 $$
 (w_1, ..., w_{n-1}, w_n) \in \mathbb B^n \mapsto (z', z_n) = ( \frac{w_1}{1 - w_n}, ..., \frac{w_{n-1}}{1 - w_n},  i \frac{1 + w_n}{1 - w_n}) \in S^{(0)}_{(0, \dots, 0, 1)}.
 $$
 
 Now, if $g \in W_b$, we can write $g = (s, a)$ accordingly to the decomposition $W_b \overset{sets}{=} U_b \times A_b$ ($U_b \cong \mathbb R$, $A_b \cong \mathbb C^{n-1}$), and we have, for any $(w', w_n) \in S_b^{(N)}$:
 \begin{equation} \label{eq:exprWb}
 	g \cdot (z', z_n) = (z' + a, z_n + i||a||^2 + 2 i \overline{a} \cdot z' + s). 
 \end{equation}
 We check easily that $S^{(N)}_b$ is preserved by $W_b$. Also, for any $g  \in L_b \simeq U(n-1) \times \mathbb R$, we can write $g = (r, t)$, and we then have
 \begin{equation} \label{eq:exprLb}
 	g \cdot (z', z_n) = (e^{t} (r \cdot z'), e^{2t} z_n).
 \end{equation}
 
 We are now ready to describe the quotients of $S_b^{(N)}$ by the action of $\Gamma' \cap N_b$. Note first that since $\Gamma'$ is neat, we have $\Gamma' \cap N_b \subset W_b$. Then, by the discussion above, we obtain a decomposition as {\em sets} $N_b \overset{sets}{=} (\mathbb  C^{n-1} \times \mathbb R) \times (U(n-1) \times \mathbb R)$, in which the elements of $\Gamma' \cap N_b$ can be written as $(a, t, {\rm Id}, 0)$. It also follows from \cite{Mok12} that $\Gamma' \cap U_b = \mathbb Z \tau$ for some $\tau \in U_b \simeq \mathbb R$. This last fact permits to form the quotient $G_b^{(N)} = \faktor{S_b^{(N)}}{U_b \cap \Gamma'}$; using \eqref{eq:exprsn}, we can also express the latter quotient as an open subset of $\mathbb C^{n-1} \times \mathbb C^\ast$:  
 $$
 G_b^{(N)} = \{ (w', w_n) \in \mathbb C^{n-1} \times \mathbb C^\ast \;\  | \;\;  |w_n| e^{\frac{2\pi}{\tau} ||w'||^2} < e^{- \frac{2\pi}{\tau}N} \},
 $$
 and the quotient is then realized by the map $(z', z_n) \in S_b^{(N)} \to (z', e^{\frac{2i\pi}{\tau} z_n}) \in G_b^{(N)}$.
 \medskip

 The group $\Lambda_b := \pi(\Gamma' \cap W_b) \subset \mathbb C^{n-1}$ is an abelian lattice, acting on $G_b^{(N)} \subset \mathbb C^{n-1} \times \mathbb C^\ast$ as
 $$
 a \cdot (z', z_n) = (z' + a, e^{-\frac{2\pi}{\tau} ||a||^2 - \frac{4\pi}{\tau} \overline{a} \cdot z'} z_n),
 $$
 Clearly, the closure $\overline{G_b^{(N)}}$ in $\mathbb C^n$ is an open neighborhood of $\mathbb C^{n-1} \times \{0\}$. We can form the quotient 
 $$
 \Omega^{(N)}_b = \faktor{\overline{G_b^{(N)}}}{\Lambda_b}
 $$
 which is then isomorphic to a tubular neighborhood of the abelian variety $\faktor{\mathbb C^{n-1}}{\Lambda_b}$ in some negative line bundle. Finally, the toroidal compactification $X'$ can be obtained by glueing the open varieties $\Omega_b^{(N)}$ to $U'$ (as $b$ runs among a system of representatives of the rational boundary components, and $N$ is large enough).
 \medskip
 
 Our claims about $X$ can be derived from the following lemma.
 \begin{lem} \label{lem:actioncomp}
 	Let $b \in \partial \mathbb B^n$ be a $\Gamma'$-rational boundary component, and let $g \in \Gamma$. Then the point $b' = g \cdot b$ is also $\Gamma'$-rational, and there exists $N, N' > 0$, for which $g$ induces an isomorphism $S_b^{(N)} \overset{g}{\to} S_{b'}^{(N')}$, yielding in turn a unique compatible biholomorphism $\Omega_b^{(N)} \to \Omega_{b'}^{(N')}$.
 \end{lem}
 \begin{proof}
 	As $\Gamma'$ is torsion free,  a point $z \in \partial \mathbb B^n$ is $\Gamma'$-rational if and only if $W_b \cap \Gamma' \neq \{ e \}$ (see \cite[\S 1.3]{Mok12}). Since $g$ normalizes $\Gamma'$, we have $g(W_b \cap \Gamma')g^{-1} \subset W_{b'} \cap \Gamma'$ so $b'$ is $\Gamma'$-rational if $b$ is.
 	
 	As for our second claim, since the set of horoballs is preserved by the action of $PU(n, 1)$, we may find $N, N'$ such that $g$ induces a isomorphism $S_b^{(N)} \to S_{b'}^{(N)}$. Let $(x', x_n)$ (resp. $(y', y_n)$) be standard coordinates on $S_b^{(N)}$ (resp. $S_b^{(N')}$) as in \eqref{eq:exprsn}, chosen so that $(y', y_n) = (x', x_n) \circ u$ for some $u \in U(n)$ satisfying $u \cdot b' = b$. Then $ug \in N_b$, and by \eqref{eq:exprWb} and \eqref{eq:exprLb}, we have $(y', y_n) \circ g = f(x', x_n)$ for some {\em affine} map $f$.
 	
 	Since $g$ normalizes $\Gamma'$, we have $g (\Gamma' \cap U_b) g^{-1} = \Gamma' \cap U_{b'}$, so the map $S_b^{(N)} \overset{g}{\to} S_{b'}^{(N')}$ passes to the quotient to give a map $\widetilde{g} : G_b^{(N)} {\to} G_{b'}^{(N')}$. Using an explicit expression for the affine map $f$, we find an (\emph{a priori} multivaluate) expression for $\widetilde{g}$ as
 	$$
 	(z', z_n) \in G_b^{(N)} \overset{\widetilde{g}}{\mapsto} (A \cdot z' + u \log z_n + z_0', \;  C\, z_n^a \, e^{b \cdot z'}) \in G_{b'}^{(N')}
 	$$
 	for some $A \in M_{n-1}(\mathbb C)$, some vectors $u, b, z_0' \in \mathbb C^{n-1}$ and $C, a \in \mathbb C$. The formula above induces a well-defined, invertible map $G_b^{(N)} \to G_{b'}^{(N')}$, so we have $u=0, a=\pm 1$. This implies that $\widetilde{g}$ extends holomorphically to $\widetilde{g} : \overline{G_b^{(N)}} \to \overline{G_{b'}^{(N')}}$. Finally, as $g$ normalizes $\Gamma'$, \(\widetilde{g}\) passes to the quotient by \(\Lambda_{b}\) and \(\Lambda_{b'}\), giving a biholomorphism $\Omega_b^{(N)} \to \Omega_{b'}^{(N')}$.
 \end{proof}

 %
 
 Going back to the proof of Proposition \ref{proptorcom}, we see that Lemma \ref{lem:actioncomp} permits to define a unique action of the quotient $G = \faktor{\Gamma}{\Gamma'}$ on $X'$, compatible with its natural action on $U'$. We then let $X:= \faktor{X'}{G}$. The following lemma ends the proof of Proposition \ref{proptorcom}, and clarifies the link with the construction of \cite{AMR10}. 
 
 \begin{lem}
 	The variety $X$ defined above does not depend on the choice of $\Gamma'$. When $\Gamma$ is arithmetic, $X$ coincides with the toroidal compactification of $U$ constructed in \cite{AMR10}.
 \end{lem}
 \begin{proof}
 	Let $\Gamma', \Gamma'' \subset \Gamma$ be two neat lattices of finite index, and let us show that the varieties constructed from $\Gamma'$ and $\Gamma''$ are the same. Since $\Gamma \cap \Gamma'$ also has finite index in $\Gamma$, we may assume $\Gamma'' \subset \Gamma'$. By Lemma~\ref{lem:actioncomp}, the action of two lattices $\Gamma'' \subset \Gamma'$ are compatible with each other over each set $\Omega_b^{(N)}$, which suffices to prove the first point. 
 	
 	Let us prove the second point. The construction of the toroidal compactification of \cite{AMR10} depends on a certain choice of {\em $\Gamma$-admissible polyhedra} for each rational boundary component (see \cite[Definition 5.1]{AMR10}). In the case of the ball, since $\dim_{\mathbb R} U_b = 1$ for any $b \in \partial \mathbb B^n$, there is only one such possible choice (cf. [{\em loc. cit.}, Theorem 4.1.(2)]). The claim now follows from the functoriality of compatible toroidal compactifications (see \cite[Lemma 2.6]{Har89}), since \enquote{choices} of polyhedra admissible for two lattices $\Gamma' \subset \Gamma$ are thus automatically compatible with each other.
 \end{proof}
 
 Note that even though this construction of $X$ is well adapted to our purposes, it should not be used to define $X$ as an {\em orbifold}, as it has the drawback of producing artificial ramification orders along the boundary components of $X$. As explained in \cite{eys18}, a better way of proceeding is to construct directly open neighborhoods of the components of $X - U$ as {\em stacks}, before glueing them to $U$.

 \subsection{Main results}
 Let us first begin with the following lemma.
 \begin{lem}\label{lem:toroidal}
 	Let $Y$ be  the toroidal compactification of the ball quotient  $U:=\faktor{\bB^n}{\Gamma}$   by a torsion free lattice $\Gamma\subset PU(n,1)$ whose parabolic isometries are all unipotent.   Let $X$ be another projective compactification of $U$, and assume that \(X\) has at most klt singularities.
 	
 	Then the identity map of $U$ extends to a birational morphism $f:X\to Y$.
 \end{lem}
 \begin{proof}
 	The identity map of   $U$ defines a birational map $f:X\dashrightarrow Y$. Assume by contradiction that $f$ is not regular. One can take a resolution of indeterminacies $\mu:\tilde{X}\to X$ for $f$ so that $\mu|_{\mu^{-1}(U)}:\mu^{-1}(U)\xrightarrow{\sim}U$ is an isomorphism:  
 	\begin{equation*}
 		\begin{tikzcd}
 			&	\tilde{X}\arrow[ld,"\mu"'] \arrow[dr,"\tilde{f}"]&  \\
 			X\arrow[rr,"f",dashed]	&&Y  
 		\end{tikzcd}
 	\end{equation*}
 	By the rigidity result (see \cite[Chapter 3, Lemma 1.15]{Deb01}), there is a fiber $\mu^{-1}(z)$  with $z\in D$ which is not contracted by $\tilde{f}$. Clearly, we have $\tilde{f}(\mu^{-1}(z)) \subset Y - U$. 
 	
 	Since $X$ has klt singularities, the work of Hacon-McKernan \cite{HM07} implies that every fiber of $\mu$ is rationally chain connected. Thus, $\widetilde{f}(\mu^{-1}(z))$ is a point since abelian varieties do not contain rational curves.
 	This gives a contradiction.
 \end{proof}
 
 \begin{rem} \label{rem:quotientsing}
 	If we make the more restrictive hypothesis that \(X\) has at most quotient singularities, we can replace the use of \cite{HM07} by the work of Kollar \cite{Kol93}, which implies that each fiber of \(\mu\) is simply connected. As $Y-U$ is a disjoint union of abelian varieties, this also implies that the image of $\tilde{f} : \mu^{-1}(z) \to Y - U$ must be a point.
 \end{rem}
 
 Let us introduce a natural class of pairs under which our rigidity theorem will hold.
 
 \begin{dfn} \label{def:quotient}
 	Let $(X, D)$ be a pair consisting of normal algebraic variety and a reduced divisor. We say that $(X, D)$ has \emph{algebraic quotient singularities} if it admits a finite affine cover $(X_i)_{i \in I}$, such that each $(X_i, D \cap X_i)$ is the quotient of a smooth SNC pair $(U_i, D_i)$ by a finite group $G_i$ leaving $D_i$ invariant.
 \end{dfn}
 
 Note that for any lattice $\Gamma \subset {\rm Aut}(\mathbb B^n)$, if $X$ is the toroidal compactification of $U = \faktor{\mathbb B^n}{\Gamma}$ described in Section \ref{sect:torcomp}, then $(X, X - U)$ has algebraic quotient singularities. 
 \medskip
 
 We can now state our main result as follows.
 
 \begin{thm}\label{thm:rigidity}
 	Let $U:=\faktor{\bB^n}{\Gamma}$  be an $n$-dimensional ball quotient  by a torsion free lattice $\Gamma\subset PU(n,1)$.  Let $X$ be a klt compactification of $U$, and let $D:=X-U$.
 	
 	Let $D^{(1)} \subset D$ be the divisorial part of $D$. If the K\"ahler-Einstein metric $\omega$ on the bundle $T_{X}(-\log D^{(1)})|_{U}$ is adapted to log order near the generic point of any component of $D^{(1)}$, then $(X, D)$ identifies with the toroidal compactification of $U$.
 \end{thm}
 
 \begin{rem} 
 	\begin{enumerate}
 		\item  Under the more restrictive assumption that $(X, D)$ has algebraic quotient singularities, the use of Lemma \ref{lem:toroidal} in our proof below can be made without appealing to the difficult result of \cite{HM07} (see Remark~\ref{rem:quotientsing}).	
 		\item  As an easy consequence of Theorem~\ref{thm:rigidity}, we can remark that there is no klt compactification $X$ of $U$ such that $X - U$ has codimension $\geq 2$. 
 	\end{enumerate}
 \end{rem}

 \begin{cor}  \label{cor:rigiditycor}
 	With the same assumptions as in Theorem \ref{thm:rigidity}, if $X$ is smooth and $D$ has simple normal crossings, then $D$ is in fact smooth, and each component is a \emph{smooth} quotient of an abelian variety $A$ by some finite group acting freely on $A$. 
 \end{cor}

 Let us prove Theorem \ref{thm:rigidity}. Let $\Gamma'\subset \Gamma$ be a subgroup of finite index so that all parabolic elements of $\Gamma'$ are unipotent. Writing $U':=\faktor{\bB^n}{\Gamma'}$, this gives a finite \'etale surjective morphism $U'\to U$. \medskip
 
 Let $X'$ be the normalization of $X$ in the function field of $U'$: this is a normal projective variety $X'$ compactifying $U'$, with a compatible finite surjective morphism $\mu:X'\to X$ (see e.g. \cite[Chapter 12, §9]{ACGH11}). 
 Since klt singularities are preserved under finite surjective morphisms, the variety $X'$ has at most klt singularities (see \cite[Corollary 5.20]{KM98}).
 
 \begin{rem}
 	If \((X, D)\) has algebraic quotient singularities, one sees easily that this is also the case for \(X'\). To see this, form the fiber product \(Z' = Z \times_{X} X'\), where \(Z \to X\) is an affine covering as in Definition~\ref{def:quotient}. By \cite[Theorem~2.23]{Kol07}, the variety \(Z'\), endowed with it natural boundary divisor, has algebraic quotient singularities. Finally, Lemma~\ref{lem:quotient2} shows that \(Z' \to X'\) is a quotient map, which gives the result. 
 \end{rem}
 
 

 Let $Y'$ be the toroidal compactification of $U'$, so that the boundary $A := Y' - U'$ is a smooth divisor.
 
 \begin{lem} \label{lem:isom}
 	The identity map on $U'$ extends as an isomorphism $f : X' \to Y'$. In particular, there is a finite surjective morphism 
 	$
 	g:Y'\to X,
 	$
 	which identifies with the \'etale and surjective map $U' \to U$ over $X-D$. 
 \end{lem}
 \begin{proof}
 	
 	Since \(X'\) is klt, \cref{lem:toroidal} shows that the identity map of $U'$ extends to a birational morphism $f:X'\to Y'$. Assume by contradiction that $f$ is not an isomorphism.	 As $Y'$ is smooth, it follows from \cite[Corollary 2.63]{KM98}  that the exceptional set ${\rm Ex}(f)$  is of pure codimension one.  Thus, the birational morphism $f$ must contract an irreducible divisorial component \(E\) of the boundary $D':= X'-U'$.  
 	\medskip
 	
 	Denote by $D^{\rm sing}$ the singular locus of $D$, and let $\omega':=\mu^*\omega$,  be the canonical K\"ahler Einstein metric on $U'$. Lemma \ref{lem:logorder} below shows that $\omega'$ is adapted to log-order for $T_{X'^\circ}(-\log E^\circ)$, where $X'^\circ:=\mu^{-1}(X-D^{\rm sing})$, and $E^\circ:=X'^\circ\cap E$.  	
 	We are going to derive a contradiction with the fact the $E$ is contracted. Let $A_1$ be the component of $A$ containing
 	$
 	f(E).
 	$ 
 	We can	 take admissible coordinates $(\mathcal{W};z_1,\ldots,z_n)$ and  $(\mathcal{U};w_1,\ldots,w_n)$ centered at some well-chosen  $x'\in E\cap X'^\circ$ and $y:=f(x')\in A_1$ respectively so that $f(\mathcal{W})\subset \mathcal{U}$, and $f|_{E}:E\to f(E)$ is smooth at $x'$. 
 	Denote by $ (f_1(z),\ldots,f_n(z))$ the expression of $f$ within these   coordinates. Then if the admissible coordinates are chosen properly, one has $$ (f_1(z),\ldots,f_n(z))=(z_1^{m_{1}}g_1(z),\ldots,z_1^{m_k}g_k(z),g_{k+1},\ldots,g_{n} )$$ where $g_1(z),\ldots,g_k(z)$ are holomorphic functions defined on $\mathcal{W}$ so that $g_i(z)\neq 0$ and $m_{i}\geq 1$ for $i=1,\ldots,k$. Since $E$ is exceptional, one has $k\geq 2$.  
 	By  the norm estimate in \cite[eq. (8) on p. 338]{Mok12}, the K\"ahler-Einstein metric $\omega$ for $T_{Y}(-\log A)|_{U}$ is adapted to log order. More precisely, one has
 	$$
 	|d w_2|_{\omega^{-1}}^2\sim (-\log |w_1|^2).
 	$$
 	Since
 	$
 	f^*d\log w_2=m_2d\log z_1+d\log g_2(w),
 	$
 	one thus has the following norm estimate
 	$$
 	|d\log z_1|_{\omega'^{-1}}^2\geq \frac{1}{m^2_2}\mu^*|d\log w_2|_{\omega^{-1}}^2-\frac{1}{m^2_2}\mu^*|\frac{dg_2}{g_2}|_{\omega^{-1}}^2\geq \frac{C(-\log |z_1|^2)}{|z_1|^{2m_2}}
 	$$
 	for some constants $C>0$. Since $d\log z_1$ is a local nowhere vanishing section for $\Omega_{X'}^1(\log D')$, we conclude that the metric $\omega'^{-1}$ for $\Omega_{X'^\circ}^{1}(-\log D'^\circ)$ is \emph{not} adapted to log order, and so is $\omega'$ for  $T_{X'^\circ}(-\log D'^\circ)$. This gives a contradiction, and ends the proof of the lemma. 
 \end{proof}
 
 \begin{lem} \label{lem:logorder}
 	With the notations of the proof of Lemma \ref{lem:isom}, the metric $\omega'$ is adapted to log-order for $T_{X'^\circ}(-\log E^\circ)$.
 \end{lem}
 \begin{proof}
 	Write $\mathcal{W}:=\mu^{-1}(\mathcal{V})$.  Since $\mu|_{\mathcal{W}-D'}:\mathcal{W}-D'\to \mathcal{V}-D$ is a finite unramified cover,  the image of $(\mu|_{\mathcal{W}-D'})_*\big(\pi_1(\mathcal{W}-D')\big)$ is a subgroup of $\pi_1(\mathcal{W}-D)\simeq \bZ$ of index $m$. Letting \(\nu(z_{1}, \cdots, z_{n}) = (z_{1}^{m}, z_{2}, \ldots, z_{n})\), one has thus the following commutative diagram
 	$$
 	\begin{tikzcd}
 		\Delta^*\times \Delta^{n-1} \arrow[d,"\nu|_{\Delta^*\times \Delta^{n-1}}"]\arrow[r, "h^\circ"] & \mathcal{W}\arrow[d, "\mu|_{\mathcal{W}}"] \\
 		\Delta^{n}\arrow[r,"\simeq"] & \mathcal{V}
 	\end{tikzcd}
 	$$
 	so that $h^\circ_{\Delta^*\times \Delta^{n-1}}:\Delta^*\times \Delta^{n-1}\to \mathcal{W}\cap U'$ is an isomorphism. By the   Riemann removable singularities theorem, $h$ extends to a holomorphic map $h:\Delta^n\to \mathcal{W}$, which is easily seen to be surjective with  finite fibers. Hence $h$ is moreover biholomorphic. $(\mathcal{W};z_1,\ldots,z_n;h)$ is therefore an admissible coordinate centered at $x'$ with $(z_1=0)=\mathcal{W}\cap D'$ so we can now identify $\mu$ with \(\nu\). 
 	Hence, $$\mu^*d\log x_1=md\log z_1, \mu^*dx_2=d z_2, \ldots, \mu^*dx_n=dz_n,$$
 	and the frame $(d\log z_1,dz_2,\ldots,dz_n)$ for $\Omega_{X'}^1(\log D')|_{\mathcal{W}}$ is adapted to log order.  This shows that the metric $\omega'$ is adapted to log order for $T_{X'^\circ}(-\log D'^\circ)$.  
 \end{proof}

 We can now conclude the case discussed in Corollary \ref{cor:rigiditycor}, where $(X, D)$ is assumed to be a smooth log-pair. Since the boundary of $Y' - U'$ is smooth, this implies that $D$ must also be smooth. Moreover,  for each connected component $A_i$ of $A$, there is a connected component $D_j$ of $D$ so that $g|_{A_i}:A_i\to D_j$ is a finite surjective morphism, which is also \'etale by the local description of $\mu$ given in the proof of Lemma \ref{lem:logorder}. Hence in this case, $D_i$ is a smooth quotient of an abelian variety by the free action of some finite group $G_i$. This suffices to establish Corollary \ref{cor:rigiditycor}.
 \medskip
 
 The proof of Theorem \ref{thm:rigidity} will be complete with the following lemma.

 \begin{lem}
 	The variety $X$ identifies with the quotient of $Y'$ by the natural action of $G = \faktor{\Gamma}{\Gamma'}$. In particular, \(X \cong Y\).
 \end{lem}
 \begin{proof}
 	This result comes right away from Lemma \ref{lem:quotient2} below, taking $M = Y'$, $N = X$, and $G = \mathcal G$. Remark that we have $R(Y')^G = R(U')^G = R(U) = R(X)$ since $U = \faktor{U'}{G}$.
 	For the second statement, remark that by Proposition \ref{proptorcom}, the toroidal compactification $Y$ of $U$ also identifies with the quotient $\faktor{Y'}{G}$. Thus, there is an isomorphism $Y \cong X$ compatible with the identity on $U$. Theorem \ref{thm:rigidity} is proved.
 \end{proof}

 \begin{lem} \label{lem:quotient2}
 	Let $f : M \to N$ be a finite surjective morphism between two normal reduced schemes. Assume that $M$ is acted upon by a finite groupoid $\mathcal G$, and that $f$ is $\mathcal G$-invariant. Suppose in addition that $R(M)^{\mathcal G} = R(N)$, where $R(M), R(N)$ are the rings of rational functions on $M, N$. Then $N$ is the quotient of $M$ by $\mathcal G$.
 \end{lem}
 \begin{proof}
 	It suffices to show that $f_\ast (\mathcal O_M)^{\mathcal G} = \mathcal O_N$. This is a local statement on the base, so we may assume that $N = {\rm Spec}\, A$, $M = {\rm Spec}\, B$, with $A$ is integrally closed. We have \(B^{\mathcal G} \subset R(B)^{\mathcal G} = R(A)\) by assumption. Since \(A \subset B\) is finite, and $A$ is integrally closed, this implies $B^{\mathcal G} \subset A$, as required.
 \end{proof}
\bigskip
\noindent \textbf{Data availability}: All data generated or analysed during this study are included in
this article.

 \providecommand{\bysame}{\leavevmode\hbox to3em{\hrulefill}\thinspace}
 \providecommand{\MR}{\relax\ifhmode\unskip\space\fi MR }
 \providecommand{\MRhref}[2]{%
 	\href{http://www.ams.org/mathscinet-getitem?mr=#1}{#2}
 }
 \providecommand{\href}[2]{#2}

\end{document}